\documentclass[twocolumn]{autart}    

\usepackage{latexsym,amssymb}
\usepackage{mathrsfs}
\usepackage{epsfig}
\usepackage{graphicx}
\usepackage{color}
\usepackage{enumerate}
\usepackage{amsmath}
\usepackage{breqn}

\newtheorem{theorem}{Theorem}
\newtheorem{corollary}{Corollary}
\newtheorem{proposition}{Proposition}
\newtheorem{assumption}{Assumption}
\newtheorem{lemma}{Lemma}
\newtheorem{remark}{Remark}
\newenvironment{definition}{\medskip\noindent{\it Definition. }}{ \medskip}

\newenvironment{proof}{\medskip\noindent{\it Proof. }}{ \medskip}
\def\begequarr{\begin{eqnarray}}
\def\endequarr{\end{eqnarray}}
\def\begequarrs{\begin{eqnarray*}}
\def\endequarrs{\end{eqnarray*}}
\def\begequ{\begin{equation}}
\def\endequ{\end{equation}}
\def\begequs{\begin{equation*}}
\def\endequs{\end{equation*}}
\def\begite{\begin{itemize}}
\def\endite{\end{itemize}}

\def\begcen{\begin{center}}
\def\endcen{\end{center}}
\def\begrem{\begin{remark}\rm}
\def\endrem{\end{remark}}
\def\ba{\begin{array}}
\def\ea{\end{array}}

\def\col{ \mbox{col}\; }

\def\dst{\displaystyle}

\def\beeq#1{\begin{equation}{#1}\end{equation}}

\begin{document}

\begin{frontmatter}

\title{Robust I\&I Adaptive Tracking Control of Systems with Nonlinear Parameterization: An ISS Perspective} 
\vspace{-2.5em}

\author[UNSD]{Lei Wang}\ead{lei.wang2@sydney.edu.au},
\author[UoN]{Christopher M. Kellett}\ead{chris.kellett@anu.edu.au}

\address[UNSD]{Australia Centre for Field Robotics, The University of Sydney, Australia.}
\vspace{-0.5em}
\address[UoN]{School of Engineering, Australian National University, Australia.}

\thanks{This work was supported by the Australian Research Council under ARC-DP160102138.}
\thanks{This work was done while L. Wang and C. Kellett were with School of Electrical Engineering and Computing, University of Newcastle, Australia.}

\vspace{-1.5em}

\begin{abstract}                          
This paper studies the immersion and invariance (I\&I) adaptive tracking problem for a class of nonlinear systems with nonlinear parameterization in the ISS framework. Under some mild assumptions, a novel I\&I adaptive control algorithm is proposed, leading to an interconnection of an ISS estimation error subsystem and an ISS tracking error subsystem. Using an ISS small-gain condition, the desired uniform global asymptotic stability of the resulting interconnected ``error" system can be achieved and a sum-type strict Lyapunov function can be explicitly constructed. Taking advantage of this ISS-based design framework, it is shown that the corresponding robustness with respect to the input perturbation can be rendered to be ISS. To remove the need to solve the immersion manifold shaping PDE, a new filter-based approach is proposed, which preserves the ISS-based design framework. Finally, we demonstrate the validness of the proposed framework on a tracking problem for series elastic actuators.
\end{abstract}

\end{frontmatter}

\section{Introduction}

Adaptive control has been a subject of significant interest over several decades \cite{Sastry&Bodson(1989)}. Several systematic design approaches have been reported in the literature, such as the Lyapunov-based method \cite{Krstic&PK(1995book),Mario&Tomei(1995book),Ioannou&Sun(2012)}, and the immersion and invariance (I\&I) method \cite{Astolfi&Karagiannis(book2008)}.  A significant difference between these approaches is that the Lyapunov-based method is certainty equivalent, while the I\&I adaptive method is noncertainty equivalent. This latter fact is due to the introduction of an extra term in the definition of the estimation error for the purpose of shaping the immersion manifold.

\vspace{-0.5em}
An adaptive controller design generally involves the design of an estimator that provides an estimate of uncertain parameters. These parameters are then used in a feedback control law such that the desired asymptotic tracking/stabilization is achieved. If extra conditions, for example a persistent excitation (PE) condition \cite{Krstic&PK(1995book)}, are satisfied, then \emph{uniform} asymptotic convergence can be concluded for the closed loop, including the asymptotic estimation of unknown parameters. Taking into account the robustness with respect to uncertainties such as input perturbations, there are several adaptive tracking design methods available in the literature for \emph{linear parameterizations}, such as $\sigma$-modification, parameter projection, or dead-zone modification and dynamic normalization \cite{Ioannou&Sun(2012)} for the Lyapunov-based method, or $\sigma$-modification \cite{Zhao&Xian(TIE2014)} or adding nonlinear damping terms \cite{Chen&Astolfi(cdc18)} for the I\&I method. With these modifications, the resulting closed-loop trajectories can be shown to be robust in the sense of boundedness with respect to a bounded perturbation, while little on the specific robustness, for example the transient performance, can be derived. From a different perspective, it is also worth noting that in \cite{Mazenc&Queiroz(TAC2013)} an explicit strict Lyapunov function was constructed for a class of unperturbed systems  by taking advantage of the cross-term between the tracking and estimation errors. With this construction, in \cite{Wang&Chen&Xu} it is shown that this strict Lyapunov function is an integral input-to-state stability (iISS) Lyapunov function with respect to the input perturbation, leading to a specific iISS robustness. Utilizing this iISS Lyapunov function, one is able to use the iISS small-gain theorem \cite{Ito&Kellett(2018)} to analyze the closed-loop stability and thus solve the so-called uncertainty propagation problem \cite{Wang&Chen&Xu}, particularly when the controlled plant is comprised of multiple interconnections. On the other hand, a similar robustness has also been reported for the I\&I method. In \cite{Wang&Chris(cdc19)} we showed that the parameter estimation error derived from the standard I\&I adaptive method \cite{Astolfi&Ortega(2003)} is iISS with respect to the tracking error, which in turn renders a simpler construction method of the strict Lyapunov function for the closed-loop system and an iISS robustness with respect to the input perturbation. However,  iISS is not an ``ideal" robustness, because the trajectories might be unbounded even with a small perturbation and the use of iISS small-gain theorem \cite{Ito&Kellett(2018)} requires the gain functions to satisfy a very restrictive condition, that  largely limits the class of systems that can be handled. In view of this, a new adaptive framework with  better robustness, such as ISS, is desirable.

\vspace{-1.1em}
The majority of the aforementioned  results on robust adaptive control can only handle linearly parameterized systems.
In order to deal with more general systems with nonlinear parameterization, extra assumptions and/or designs are required in general.
The global adaptive control problem for nonlinearly parameterized systems was solved in \cite{Marino&Tomei(1993)} via output feedback by designing an adaptive gain parameter.  In \cite{Annaswamy&Skantze(1998)}, utilizing the min-max optimization procedure, a new adaptive controller was proposed for systems with a convex/concave parameterization such that the tracking errors  converge to an arbitrarily small set. In \cite{Ge1999} for a class of nonlinear parameterizations, a Lyapunov-based controller and parameter updating law was developed by employing a Lyapunov function in an integral form. In \cite{Liu&Ortega(2010TAC)} for a nonlinear parameterization fulfilling some strict global monotonicity properties, the I\&I adaptive method is used to generate an estimation error subsystem whose zero equilibrium point is globally asymptotically stable.
See also \cite{Wang&Ortega(2015TAC),Lin(2002)} for other interesting progress on adaptive control with nonlinear parameterization.
In spite of these impressive results, it is worth stressing that the robust adaptive control problem with nonlinear parameterization is generally complicated and in fact, somewhat poorly understood.

\vspace{-0.5em}
In this paper, a novel I\&I adaptive tracking control method is developed for a class of nonlinearly parameterized systems in the ISS framework  \cite{Sontag-SCL1998-Integral-Variants,Kellett&Wirth(TAC2016),Sontag(bookchapter2008),Sontag(tac1989)}. Using a (cyclic) small-gain theorem \cite{Ito(CDC2002),LiuHillJiang(Aut2011)}, we demonstrate the desired uniform global asymptotic stability and robustness (i.e., ISS) with respect to the input perturbation. More explicitly, under a \emph{local} monotonicity-like assumption, the standard I\&I adaptive control algorithm is modified by employing the vector saturation function and the dead-zone function so as to yield an ISS estimation error system. If the feedback law is appropriately designed in such a way that the tracking error subsystem is also ISS, and an ISS small-gain condition is satisfied, then the desired uniform global asymptotic stability of the resulting interconnected ``error" system can be achieved and a sum-type strict Lyapunov function can be explicitly constructed. Taking advantage of this ISS-based design framework, we also show that the corresponding robustness with respect to the input perturbation can be rendered to be ISS.

\vspace{-0.5em}
As in the standard I\&I adaptive design \cite{Astolfi&Ortega(2003)}, the immersion manifold shaping PDE needs to be solved to define the estimation error.
The conventional method to remove this constraint is to add a nonlinear dynamic scaling \cite{Liu&Ortega(2011SCL)}, resulting in a dynamic high-gain parameter. It is noted that this high-gain parameter dynamics is fundamentally an unstable system, which prevents the construction of a strict Lyapunov function and thus impedes the robustness analysis for the closed loop. In view of this, we  introduce a filter whose state is used to re-define the estimation error. Consequently, the closed-loop error systems can be transformed into an interconnection of three ISS subsystems in a lower-triangular structure, for which the controller design and stability analysis in the ISS framework can be performed.


\vspace{-0.5em}
This paper is organized as follows. In Section 2 some useful notations and definitions are given and the considered adaptive problem is formulated. For  unperturbed systems, Section 3  presents an ISS-based I\&I adaptive controller design framework, whose robustness with respect to the input perturbation is later analyzed in Section 4. In Section 5, a filter-based approach is given to remove the constraint of solving the immersion manifold shaping PDE. To demonstrate the effectiveness of our approach, the adaptive tracking problem for series elastic actuators is investigated in Section 6. Finally, a brief conclusion is made in Section 7. All technical proofs are presented in the Appendix. Compared to the preliminary version \cite{Wang&Chris(cdc19)},  this paper has improved the design of feedback control in Section 3.2, and further developed the contexts in Sections 4, 5, and 6.
\vspace{-0.5em}

\section{Preliminaries}
\vspace{-0.5em}
\subsection{Notations and definitions}

\vspace{-0.8em}
In this paper, the controlled system will be augmented with a parameter estimator leading to an interconnected system, for which  ISS-Lyapunov functions and  ISS-small-gain theorems can be used  for the controller design and stability analysis. To make the paper self-contained, in the following we present some useful notations, the definition of an ISS-Lyapunov function, and the cyclic-small-gain theorem.

\vspace{-0.5em}
{\bf Notation:} Let $\mathbb{B}_r^q$ denote the ball $\{x\in\mathbb{R}^q: |x|\leq r\}$. A continuous function $\rho:\mathbb{R}_+:=[0,\infty)\rightarrow\mathbb{R}_+$ is said to be of class $\mathcal{PD}$ if $\rho$ is positive definite. A continuous function $\rho:\mathbb{R}_+\rightarrow\mathbb{R}_+$ is said to be of class $\mathcal{SN}$ if $\rho$ is nondecreasing. A function $\rho\in\mathcal{SN}$ is said to be of class $\mathcal{K}$ if $\rho$ is strictly increasing and $\rho(0)=0$. A class $\mathcal{K}$ function is of class $\mathcal{K}_{\infty}$ if it is unbounded. The symbol $\mbox{Id}$ denotes the identity function on $\mathbb{R}_+$.
For a continuous map $\rho:\mathbb{R}_+\rightarrow\mathbb{R}_+$, the map $\rho^{\ominus}:\mathbb{R}_+\rightarrow[0,\infty]$ represents $\rho^{\ominus}(s)=\mbox{sup}\{r\in\mathbb{R}_+:\rho(r)\leq s\}$. Notice that, given a function $\rho\in\mathcal{K}$, by definition, $\rho^{\ominus}(s)=\infty$ holds for all $s\geq\lim_{r\rightarrow\infty}\rho(r)$, and $\rho^{\ominus}(s)=\rho^{-1}(s)$ elsewhere. It is also noted that, in the case of $\rho\in\mathcal{K}\backslash \mathcal{K}_\infty$, we have $\rho\circ\rho^\ominus \leq \mbox{Id}$ and  $\rho^\ominus\circ\rho = \mbox{Id}$. For a continuously differential function $V(t,x(t))$, we denote $\dot V(t,x(t))$ as $\dfrac{d V}{d t}(t,x(t)):=\dfrac{\partial V}{\partial t}(t,x(t)) + \dfrac{\partial V}{\partial x}\dfrac{{d} x(t)}{{d} t}$.

\vspace{-0.5em}We consider the networked time-varying system
\beeq{\label{sys-x-i}
\dot x_i = f_i(t,x,u)\,\quad i=1,\ldots,N
}
where $x=\col(x_1,\ldots,x_N)$ with $x_i\in\mathbb{R}^{n_i}$, and $u\in\mathbb{R}^m$.

\vspace{-0.8em}
\begin{definition}
A continuously differentiable function $V_i:\mathbb{R}_+\times\mathbb{R}^{n_i}\rightarrow \mathbb{R}_+$ is a \emph{uniform} ISS-Lyapunov function for (\ref{sys-x-i}) with respect to inputs $x_j$, $j\neq i$ and $u$ if there exist $\underline{\alpha}_{i},\bar\alpha_i, \alpha_i\in\mathcal{K}_\infty$ and $\sigma_{ij},\sigma_{iu}\in\mathcal{K}$ such that
\begin{subequations}
\begin{align}
&\underline{\alpha}_{i}(|x_i|) \leq V_{i}(t,x_i) \leq \bar\alpha_i(|x_i|)\, \label{P-V-i-a}\\
&V_i(t,x_i)\geq \max_{j=1,\ldots,N;j\neq i}\{\sigma_{ij}(|x_j|),\sigma_{iu}(|u|)\} \, \Longrightarrow\,\nonumber\\
&\qquad \qquad\dst\frac{\partial V_{i}}{\partial t}+\dst\frac{\partial V_{i}}{\partial x_i}f_i(t,x,u) \leq -\alpha_i(|x_i|)\,. \label{P-V-i-b}
\end{align}
\end{subequations}
Then the $x_i$-subsystem (\ref{sys-x-i}) is said to be \emph{uniform} ISS with respect to state $x_i$ and inputs $u$, $x_j$, $j\neq i$ if (\ref{sys-x-i}) admits such a \emph{uniform} ISS-Lyapunov function.
\end{definition}

\vspace{-0.8em}
By setting $\gamma_{ij}=\sigma_{ij}\circ\underline{\alpha}^{-1}_{j}$, $i\neq j$, (\ref{P-V-i-b}) can be replaced by
\[\ba{l}
V_i(t,x_i)\geq \max\limits_{j=1,\ldots,N;j\neq i}\{\gamma_{ij}(|V_j|),\sigma_{iu}(|u|)\} \, \Longrightarrow\,\\
\qquad \qquad\dst\frac{\partial V_{i}}{\partial t}+\dst\frac{\partial V_{i}}{\partial x_i}f_i(t,x,u) \leq -\alpha_i(|x_i|)\,.
\ea
\]
Using the terminology of \cite{SontagWang(1995SCL)} or \cite{Kellett&Wirth(TAC2016)}, an ISS Lyapunov function satisfying (\ref{P-V-i-a})-(\ref{P-V-i-b}) is the implication form. In addition, there are other commonly used forms, such as the dissipation form (see \cite{Kellett&Wirth(TAC2016)} for more explicit discussions on their connections).

\vspace{-0.5em}
To study the stability of the network of $N$ uniform ISS subsystems (\ref{sys-x-i}), there are several methods reported in the literature, such as \cite{LiuHillJiang(Aut2011),Ito&Jiang(TAC2013),Dash&Ruffer(SIAM2010)}.
To ease the subsequent design and analysis, in this paper  the cyclic-small-gain theorem \cite{LiuHillJiang(Aut2011)} will be adapted to the current \emph{time-varying} setting (see Theorem \ref{theo-csgt} below), though it is unclear how to construct the corresponding \emph{smooth} ISS Lyapunov function for the network, in general.

\begin{theorem}\label{theo-csgt}
Consider the time-varying system (\ref{sys-x-i}). Suppose every $i$-th subsystem admits a uniform ISS-Lyapunov function $V_i(t,x_i)$ satisfying (\ref{P-V-i-a}) and (\ref{P-V-i-b}). Then the system (\ref{sys-x-i}) is uniformly ISS with respect to the state $x$ and the input $u$ if for every simple cycle $(\mathcal{V}_{i_1}\mathcal{V}_{i_2}\cdots \mathcal{V}_{i_r}\mathcal{V}_{i_1})$
\beeq{\label{csmc}
\gamma_{i_1i_2}\circ \gamma_{i_2i_3}\circ\cdots\circ\gamma_{i_ri_1} < \mbox{Id}
}
where $r=2,\ldots,N$ and $1\leq i_j\leq N$, $i_j\neq i_{j'}$ if $j\neq j'$. $\square$
\end{theorem}

%

\subsection{Problem Formulation}
\label{sec-profor}

\vspace{-0.5em}
Consider adaptive tracking control of nonlinear systems of the form
\beeq{\label{sys-x}
\dot x = f_1(x) + \phi(\theta,x) + g_1(x)u
}
where state $x\in\mathbb{R}^n$, control $u\in\mathbb{R}^m$, and uncertain parameter $\theta\in\Theta\subset\mathbb{R}^q$ with $\Theta$ being a known compact set. Throughout this paper, we make the following assumptions on systems (\ref{sys-x}).
\vspace{-0.5em}
\begin{assumption}\label{ass-1}
The function $\phi(\theta,x)$ is \ $C^1$ (i.e.,continuously differentiable) in $\theta\in\mathbb{R}^q$, for all $x\in\mathbb{R}^n$.
\end{assumption}
\vspace{-0.5em}
\begin{assumption}\label{ass-2}
The tracking reference $x_r(t)$ satisfies the following properties:
\vspace{-1em}
  \begin{itemize}
    \item There exists a constant $r_1>0$ such that $x_r(t)\in\mathbb{B}_{r_1}^n$ holds for all $t\geq0$;
    \item $x_r(t)$ is $C^1$;
    \item $x_r(t)$ and $\dot x_r(t)$ are known a priori.
  \end{itemize}
\end{assumption}

It is noted that the function $\phi(\theta,x)$ is nonlinearly parameterized, with the linear parameterization as a special case.  To ease the subsequent analysis, let $l_\theta$ be
\beeq{\label{ell}
 l_\theta  \geq \max_{\theta\in\Theta} |\theta|\,.
}

\vspace{-1.8em}
Defining the tracking error as $e:=x-x_r$, we have
\beeq{\label{sys-e}
\dot e = f_1(e+x_r) + \phi(\theta,e+x_r) + g_1(e+x_r)u - \dot x_r(t)\,.
}
In this way, the problem at hand becomes one of adaptive stabilization of the \emph{nonautonomous} system (\ref{sys-e}), in the presence of the nonlinear parameterization.
Before we proceed to the explicit design, we make a few important observations.

\vspace{-0.5em}
Define
\vspace{-0.5em}
\beeq{\label{def-tphi}\vspace{-1.5em}
\tilde\phi_s(\theta^\dag,\theta^\ddag,x) := \phi(\mbox{satv}(\theta^\dag+\theta^\ddag),x)-\phi(\mbox{satv}(\theta^\dag),x)\,
}
where $\theta^\dag,\theta^\ddag\in\mathbb{R}^q$, and the function $\mbox{satv}(\cdot)$ is defined as
\vspace{-0.5em}
\[\vspace{-0.5em}
\mbox{satv}(\col(s_1,\ldots,s_q)) = \col(\mbox{sat}(s_1),\ldots,\mbox{sat}(s_q))
\]
with $\mbox{sat}(\cdot)$ denoting a smooth nondecreasing saturation function of the form
\vspace{-0.5em}
\beeq{\label{satv}
\mbox{sat}(s)=
\left\{\ba{ll} s \,,\qquad &|s|\leq l_s \\
s - \mbox{sign}(s)\dst\frac{(|s|-l_s)^2}{2\epsilon_s} \,, &l_s\leq|s|\leq l_s+\epsilon_s\, \\
(l_s+ 0.5\epsilon_s)\mbox{sign}(s) \,,\quad & |s|\geq l_s+\epsilon_s\, \ea\right.
\vspace{-1.5em}
}
in which the saturation level $l_s> l_\theta$ and the margin constant $1\geq\epsilon_s>0$.

\vspace{-0.5em}
With the above definition of the function $\mbox{satv}(\cdot)$, it can be seen that $\theta=\mbox{satv}(\theta)$ for all $\theta\in\Theta$, which yields that there always exists $\gamma_s\in\mathcal{K}\backslash \mathcal{K}_\infty$, satisfying
\vspace{-0.5em} \beeq{\label{L-gamma}
\gamma_s(s) \leq l_{\gamma}\,,\qquad \gamma_s(s) \leq s\,,\quad\forall s\in\mathbb{R}_+\,
\vspace{-0.5em}
}

\vspace{-1.5em} \noindent
with a constant $l_{\gamma}>0$, such that for all $\tilde\theta\in\mathbb{R}^q$, and $\theta\in\Theta$
\beeq{\label{gamma-1}
|\mbox{satv}(\theta+\tilde\theta)-\theta| \leq \gamma_s(|\tilde\theta|)\,.
\vspace{-0.5em}}

\vspace{-1.3em}
Let function $\kappa_1\in\mathcal{SN}$ be such that
\[
\dst\kappa_1(s) \geq \sup_{(\theta,x_r,e)\in\mathbb{B}_{l_s+0.5\epsilon_s}^q\times\mathbb{B}_{r_1}^n\times\mathbb{B}_{s}^n}\left|\frac{\partial\phi(\theta,e+x_r)}{\partial \theta}\right|\,.
\]
Thus we have
\beeq{\label{Ineq-UP}
|\tilde\phi_s(\theta^\dag,\theta^\ddag,e+x_r)| \leq \kappa_1(|e|)\gamma_s(|\theta^\ddag|)\,,
}
for all $\theta^\dag\in\Theta$ and $\theta^\ddag\in\mathbb{R}^q$.


\section{Uniform Global Asymptotic Stability}
\label{sec-MR}

In this section, a new I\&I adaptive controller design paradigm will be proposed for the nonlinearly parameterized system (\ref{sys-e}). With the construction of an ISS interconnected system, we will show that the resulting closed-loop ``error" system is uniformly globally asymptotically stable at the origin.

\subsection{The parameter estimator design}
\vspace{-0.5em}
With the tracking error system (\ref{sys-e}), following the I\&I adaptive design approach \cite{Astolfi&Ortega(2003)}, the estimation error is defined by
\beeq{\label{tilde-theta}
\tilde\theta: = \hat\theta - \theta + \beta(e,x_r)
}

\vspace{-1.6em}\noindent
where $\hat\theta$ denotes the state of the parameter estimator, and the function $\beta(e,x_r)$ is an extra term that will be designed to shape the manifold into which the adaptive system
will be immersed.
The derivative of the estimation error $\tilde\theta$ is then given by
\beeq{\label{tilde-theta-0}\ba{rcl}
\dot{\tilde\theta} = \dot{\hat\theta}  + \dst\frac{\partial\beta}{\partial e}[f_1(x) + \phi(\theta,x) + g_1(x)u - \dot x_r]  + \dst\frac{\partial\beta}{\partial x_r}\dot x_r\,.\\
\ea}

\vspace{-1.5em}
We design the parameter estimator as
\vspace{-1.5em}
\begin{dmath}\label{hat-theta}
\dot{\hat\theta} =  - \dst\frac{\partial\beta}{\partial e}\left[f_1(x) + \phi(\mbox{satv}(\hat\theta+\beta(e,x_r)),x) + g_1(x)u - \dot x_r\right]
 -\dst\frac{\partial\beta}{\partial x_r}\dot x_r -k_{dz}\,\mbox{dzv}(\hat\theta+\beta(e,x_r))
\end{dmath}\vspace{-0.5em}

\vspace{-1.2em}\noindent
in which $k_{dz}>0$ is a design parameter to be fixed later,
and the function $\mbox{dzv}(\cdot)$ is defined as
\[
\mbox{dzv}(\col(s_1,\ldots,s_q)) = \col(\mbox{dz}(s_1),\ldots,\mbox{dz}(s_q))
\]
with $\mbox{dz}(\cdot)$ denoting a smooth dead-zone function of the form
\beeq{\label{def-dz}
\mbox{dz}(s)=\left\{\ba{l} 0 \,,\qquad |s|\leq  l_\theta   \\ (|s|- l_\theta  )^2\left[2( l_\theta  +1)^2-(2 l_\theta  +1)|s|\right]\mbox{sign}(s)\,,\\ \qquad \quad  l_\theta   < |s| <  l_\theta  +1\,\\ s \,,\qquad |s|\geq  l_\theta  +1\, \ea\right.
}
and the dead-zone amplitude $ l_\theta  $ given in (\ref{ell}).

\begin{remark}
Compared with the conventional I\&I estimator \cite{Astolfi&Karagiannis(book2008),Liu&Ortega(2010TAC)}, (\ref{hat-theta}) additionally employs the vector saturation function and the vector deadzone function, which as shown later enables us to obtain a uniform ISS estimation error system, with a mild local assumption. $\square$
\end{remark}

Substituting (\ref{hat-theta}) into  (\ref{tilde-theta-0}) yields
\beeq{\label{tilde-theta-1}
\dot{\tilde\theta} = -\dst\frac{\partial\beta}{\partial e}\tilde\phi_s(\theta,\tilde\theta,x)-k_{dz}\,\mbox{dzv}(\theta+\tilde\theta)\,,
}
with $\tilde\phi_s$ defined by (\ref{def-tphi}).

In the following, we will analyze the stability of the resulting estimation error system (\ref{tilde-theta-1}), based on the following assumption.
\begin{assumption}\label{ass-PE-2}
 There exist $l_s > \sqrt{q} \,l_\theta  $ and a continuous matrix-valued function $\varsigma:\mathbb{R}^n\rightarrow\mathbb{R}^{q\times n}$  such that
  \beeq{\ba{r}\label{ineq-ND}
 (\theta'-\theta)^{\top}\varsigma(x_r)[\phi(\theta',x_r)-\phi(\theta,x_r)]\geq \qquad \qquad \\ (\theta'-\theta)^{\top}M_1(x_r)(\theta'-\theta)\,\geq\, 0\,
    \ea}
holds for all $\theta\in\Theta$ and $|\theta'|\leq l_s$, with some continuous function $M_1:\mathbb{R}^n\rightarrow\mathbb{R}^{q\times q}$ satisfying
        \beeq{\label{ineq-PE}
            \int_{t}^{t+\delta} M_1(x_r(\tau))+M_1^{\top}(x_r(\tau)) d \tau \geq \mu  I\,
        }
for  some constants $\delta>0$ and $\mu >0$, and all $t\geq0$.
\end{assumption}
\begin{remark}
The inequality (\ref{ineq-ND}) in fact demonstrates a property of function $\phi(\theta,x)$ for $x$ being in its steady state and $\theta$ being in a neighborhood of the ball of radius $l_\theta$. Namely, for all $x\in\mathbb{B}_{r_1}^n$, there exists a function $\varsigma(x)$ such that the function $\varsigma(x)\phi(\theta,x)$ is non-decreasing in $\theta\in\mathbb{B}^q=:\{s\in\mathbb{R}^q: |s|\leq  l_\theta +\varepsilon_0\}$ with some $\varepsilon_0>0$.
This local condition is  weaker than that in \cite{Liu&Ortega(2010TAC)}, where a \emph{strict} increase is required for all $\theta\in\mathbb{R}^q$ and $x\in\mathbb{R}^n$. As a particular case, (\ref{ineq-ND}) can always be satisfied if $\phi(\cdot)$ is linearly parameterized.
Inequality (\ref{ineq-PE}) guarantees that the estimator (\ref{hat-theta}) is persistently excited so as to achieve a uniform asymptotic estimation. $\square$
\end{remark}
\vspace{0.5em}

With Assumption \ref{ass-PE-2} in mind, we choose $\beta(e,x_r)$ as
\beeq{\label{beta}
\beta(e,x_r) = \int_{0}^{e} \varsigma(s+x_r)ds
}
with the function $\varsigma$ defined in Assumption \ref{ass-PE-2}.
\begin{remark}
It is noted that to derive the explicit expression of $\beta$, we need to solve the PDE
\beeq{\label{PDE}
\frac{\partial {\beta}(s,x_r)}{\partial s} =\varsigma(s+x_r)\,
}
whose solvability in a general sense is not guaranteed.
This limitation, however, can be overcome by introducing an extra filter, which will be detailed in Section \ref{sec-PDE}.
$\square$
\end{remark}
\vspace{0.5em}
With  (\ref{beta}), the estimation error system (\ref{tilde-theta-1}) reduces to
\beeq{\label{tilde-theta-2}
\dot{\tilde\theta} = H(\theta,\tilde\theta,x_r) + \Delta(\theta,\tilde\theta,x_r,e)
}
where for compactness we define
\beeq{\label{def-H}
H(\theta,\tilde\theta,x_r)=-\varsigma(x_r)\tilde\phi_s(\theta,\tilde\theta,x_r)-k_{dz}\mbox{dzv}(\tilde\theta+\theta)
}
and
$$
\Delta(\theta,\tilde\theta,x_r,e)  = \varphi(\mbox{satv}(\theta),x_r,e+x_r)-\varphi(\mbox{satv}(\theta+\tilde\theta),x_r,e+x_r)
$$
with
$
\varphi(\theta^\dag,x_r,e^\dag)=\varsigma(x_r)\phi(\theta^\dag,x_r)-\varsigma(e^\dag)\phi(\theta^\dag,e^\dag)\,.
$

An instrumental property of function $\varphi$ is formulated as below, with the proof  given in Appendix \ref{app-varphi}.
\begin{lemma}\label{lemma-varphi}
Suppose that Assumptions \ref{ass-1} and \ref{ass-2} hold. There exists a function $\kappa_2\in\mathcal{K}$ such that
\beeq{\label{bound-varphi}
|\varphi(\theta^\dag,x_r,e+x_r)-\varphi(\theta^\ddag,x_r,e+x_r)| \leq \kappa_2(|e|)|\theta^\dag-\theta^\ddag|\,
}
for all $x_r\in\mathbb{B}^n_{r_1}$, $e\in\mathbb{R}^n$ and $\theta^\dag,\theta^\ddag\in\mathbb{B}^q_{l_s+0.5\epsilon_s}$. $\square$
\end{lemma}
\vspace{0.5em}
With the above lemma, recalling (\ref{L-gamma}) and (\ref{gamma-1}), we have
\beeq{\ba{rcl}\label{bound-Delta-2}
|\Delta(\theta,\tilde\theta,x_r,e)|&\leq& \kappa_2(|e|)\gamma_s(|\tilde\theta|)\leq l_{\gamma}\kappa_2(|e|)\,.
\ea}

We now proceed to study the stability property of the estimation error system (\ref{tilde-theta-2}), and consider the \emph{nonautonomous auxiliary system} of the form
\beeq{\label{auxi-tilde-theta-2}
\dot{\tilde\theta} = H(\theta,\tilde\theta,x_r(t))
}
whose stability property is formulated as below, with the proof given in Appendix \ref{app-1}.

\begin{lemma}\label{lemma-1}
Suppose that Assumptions \ref{ass-1}--\ref{ass-PE-2} holds. Then there exists a $k_{dz}^\star>0$ such that for all $k_{dz}\geq k_{dz}^\star$, the origin of system (\ref{auxi-tilde-theta-2}) is uniformly globally exponentially stable with  a $C^1$ function $V_{est}:\mathbb{R}_{+}\times\mathbb{R}^{q}\rightarrow\mathbb{R}_{+}$ and  constants $a_i>0$, $1\leq i\leq 4$ such that
\begin{subequations}
\begin{align}
&a_1|\tilde\theta|^2 \leq V_{est}(t,\tilde\theta) \leq a_2|\tilde\theta|^2\, \label{Vest-a}\\
&\dst\frac{\partial V_{est}}{\partial t}+\frac{\partial V_{est}}{\partial \tilde\theta}H(\theta,\tilde\theta,x_r) \leq -a_3 |\tilde\theta|^2\, \label{Vest-b}\\
&\left|\dst\frac{\partial V_{est}}{\partial \tilde\theta}\right|\leq a_4 |\tilde\theta|\,\label{Vest-c}.
\end{align}
\end{subequations}
\end{lemma}

Bearing in mind the significant property of (\ref{auxi-tilde-theta-2}) addressed in Lemma \ref{lemma-1}, we turn to consider the actual estimation error system (\ref{tilde-theta-2}).

\begin{lemma}\label{lemma-non-ex}
Suppose that Assumptions \ref{ass-1}--\ref{ass-PE-2} hold. Then system (\ref{tilde-theta-2}) is uniformly ISS with respect to state $\tilde\theta$ and input $e$, with the uniform ISS Lyapunov function $V_{est}(t,\tilde\theta)$ fulfilling (\ref{Vest-a}) and (\ref{Vest-c}), and
\beeq{\label{ineq-V-2-b}
V_{est} \geq \sigma_{\tilde\theta,e}(|e|) \,\Longrightarrow\, \dot V_{est}(t,\tilde\theta)\leq -\frac{a_3(\tau_{est}-1)}{\tau_{est}}|\tilde\theta|^2\,
}
for any $\tau_{est}>1$, and
\begin{eqnarray}
\sigma_{\tilde\theta,e}(s)&=&a_1(\tau_{est}a^\ast l_{\gamma})^2\kappa_2(s)^2\,\label{sigma-21}\\
a^\ast &=& \dst\frac{\sqrt{a_2}a_4}{\sqrt{a_1}a_3}\, \label{a-ast}\,.
\end{eqnarray}
\end{lemma}
\begin{proof}
It is clear from Lemma \ref{lemma-1} that (\ref{Vest-a}) and (\ref{Vest-c}) are satisfied. As for the proof of uniform ISS stability of system (\ref{tilde-theta-2}), we take the time derivative of $V_{est}(t,\tilde\theta)$ along (\ref{tilde-theta-2}), which using (\ref{Vest-b}) and (\ref{bound-Delta-2}), yields
\[\ba{rcl}
\dot V_{est} &=& \dst\frac{\partial V_{est}}{\partial t}+\frac{\partial V_{est}}{\partial \tilde\theta}H(\theta,\tilde\theta,x_r)  + \dst\frac{\partial V_{est}}{\partial \tilde\theta} \Delta(\theta,\tilde\theta,x_r,e)\,\\
&\leq& -a_3|\tilde\theta|^2 + a_4l_{\gamma} |\tilde\theta|\kappa_2(|e|)\,\\
&=& -\dfrac{a_3(\tau_{est}-1)}{\tau_{est}}|\tilde\theta|^2 - \dfrac{a_3}{\tau_{est}}|\tilde\theta|\left(|\tilde\theta|- \dfrac{\tau_{est}a_4l_{\gamma}}{a_3}\kappa_2(|e|)\right)
\ea\]
with any $\tau_{est}>1$.
Then, it immediately follows that
\[
|\tilde\theta| \geq \dfrac{\tau_{est}a_4l_{\gamma}}{a_3}\kappa_2(|e|) \,\Longrightarrow\, \dot V_{est}\leq -\frac{a_3(\tau_{est}-1)}{\tau_{est}}|\tilde\theta|^2\,,
\]
which leads to (\ref{ineq-V-2-b}) by recalling the right side of (\ref{Vest-a}).
This completes the proof. $\square$
\end{proof}

\begin{remark}
It is observed from (\ref{ineq-V-2-b}) and (\ref{sigma-21}) that the ISS gain function of the estimation error system (\ref{tilde-theta-2}) strongly relies on the function $\kappa_2$, which in fact can be shaped by the design freedom $\varsigma(\cdot)$ as seen from Lemma \ref{lemma-varphi}.
\end{remark}


\subsection{The control feedback design}

With the proposed estimator (\ref{hat-theta}), and bearing in mind the uniform ISS property of system (\ref{tilde-theta-2}), we now proceed to design the feedback law $u$ for (\ref{sys-e}).

A minimum requirement for successful tracking is that the controlled system (\ref{sys-e}) with a known $\theta$ is stabilizable. In this paper we make the following explicit stabilizability assumption.
\begin{assumption}\label{ass-ideal-sys}
There exists a function $\psi(x_r,\dot x_r,\theta,e)$ such that the zero equilibrium point of system
\beeq{\label{ideal-sys-e}
\dot e = F(x_r,\dot x_r,\theta,e)\,
}
with
\[\ba{l}
F(x_r,\dot x_r,\theta,e)= f_1(e+x_r) + \phi(\theta,e+x_r) \qquad  \qquad\,\\ \qquad \qquad \qquad\qquad + g_1(e+x_r)\psi(e,x_r,\dot x_r,\theta) - \dot x_r(t)\,,
\ea\]
is globally asymptotically stable, uniformly in $t\in\mathbb{R}_+$ and $\theta\in\mathbb{R}^q$. More specifically, there exist a strict Lyapunov function $V_{err}(t,e)$,  and class $\mathcal{K}_{\infty}$ functions $\alpha_i$, $1\leq i\leq4$ such that for all $\theta\in\mathbb{R}^q$
\begin{subequations}
\begin{align}
&\alpha_1(|e|) \leq V_{err}(t,e) \leq \alpha_2(|e|)\, \label{ineq-V-a}\\
&\dst\frac{\partial V_{err}}{\partial t}+\dst\frac{\partial V_{err}}{\partial e}F(e,x_r,\dot x_r,\theta) \leq -\alpha_3(|e|)\,\label{ineq-V-b}\\
&\dst\left|\frac{\partial V_{err}}{\partial e}\right| \leq \alpha_4(|e|)\, \label{ineq-V-c}
\end{align}
\end{subequations}
are satisfied. $\square$
\end{assumption}
\vspace{0.8em}
\begin{remark}
To satisfy Assumption \ref{ass-ideal-sys}, system (\ref{sys-e}) generally needs to satisfy a matching condition so as to guarantee the existence of such a Lyapunov function $V_{err}$, independent of $\theta$. In spite of this, the proposed method can also be extended to some unmatched cases by employing techniques, such as backstepping (see the subsequent Section \ref{sec-SEAs} for an example with an explicit construction). $\square$
\end{remark}

With Assumption \ref{ass-ideal-sys}, we design the control law $u$ as
\beeq{\label{u-1}
u=\psi(x_r,\dot x_r,\mbox{satv}(\hat\theta + \beta(e,x_r)),e)\,.
}
Substituting (\ref{u-1}) into (\ref{sys-e}) yields
\beeq{\label{err-sys}
\dot e = F(x_r,\dot x_r,\mbox{satv}(\hat\theta + \beta(e,x_r)),e)  - \tilde\phi_s(\theta,\tilde\theta,e+x_r),
}
along which the time derivative of $V_{err}$ is computed by
\begin{eqnarray}
\dst\dot V_{err}
&\leq& -\alpha_3(|e|) + \left|\dst\frac{\partial V_{err}}{\partial e}\right| \,|\tilde\phi_s(\theta,\tilde\theta,e+x_r)|\, \nonumber\\
 &\leq& -\alpha_3(|e|) + \alpha_4(|e|)\kappa_1(|e|)\gamma_s(|\tilde\theta|)\, \label{V-err-1}
\end{eqnarray}
where (\ref{ineq-V-b}) is used to obtain the first inequality, and (\ref{ineq-V-c}) and (\ref{Ineq-UP}) are used to obtain the last inequality. Note that $\kappa_1\in\mathcal{SN}$ and $\gamma_s\in\mathcal{K}\backslash \mathcal{K}_\infty$.

The function $\alpha_3$ in (\ref{ineq-V-b}) can be shaped by appropriately designing the ``ideal" control $u=\psi(\cdot)$. Inequality (\ref{V-err-1}) suggests appropriately designing $\psi(\cdot)$ to obtain $\alpha_3(\cdot)$ so that  system (\ref{err-sys}) is uniformly ISS with respect to state $e$ and input $\tilde\theta$. The resulting closed-loop system (\ref{tilde-theta-2}), (\ref{err-sys}) can then be viewed as an interconnection of two uniform ISS subsystems, for which the standard ISS small-gain theorem \cite{Jiang&Teel(1994)} or Theorem \ref{theo-csgt} with $N=2$ can be employed to verify the closed-loop asymptotic stability. In view of these intuitions, the following theorem is concluded, with the proof given in Appendix \ref{app-sec-proof-Th2}.

\begin{theorem}\label{theo-general}
Consider system (\ref{sys-x}) with parameter estimator (\ref{hat-theta}) and feedback controller (\ref{u-1}). Suppose that Assumptions \ref{ass-1}--\ref{ass-ideal-sys} hold, and there exists a constant $\tau_{err}>1$ such that
\beeq{\label{alpha-3}
\alpha_3(s)\geq \tau_{err}\gamma_s\circ a^\ast l_{\gamma}\kappa_2\circ\alpha_1^{-1}\circ\alpha_2(s)\cdot \alpha_4(s)\cdot\kappa_1(s)\,
}
for all $s\in\mathbb{R}_+$.
Then the zero equilibrium point of the resulting closed-loop error system (\ref{tilde-theta-2}), (\ref{err-sys}) is uniformly globally asymptotically stable.  $\square$
\end{theorem}
\begin{remark}
It is observed that the proposed design paradigm consists of two design freedoms: $\varsigma(\cdot)$ and $u$, that can be used to shape functions $\kappa_2$ and $\alpha_3$, respectively, such that (\ref{alpha-3}) is satisfied. $\square$
\end{remark}

In Theorem \ref{theo-general}, following the ISS small-gain theorem \cite{LiuHillJiang(Aut2011)}, we present a sufficient condition of uniformly globally asymptotically stabilizing the origin of the interconnected system (\ref{tilde-theta-2}), (\ref{err-sys}). However, as shown in \cite{LiuHillJiang(Aut2011)} it is unclear how to construct the corresponding \emph{smooth} Lyapunov function, which can play a significant role in analyzing the system performance and dealing with other problems such as adaptive output regulation. In the following, inspired by the idea of \cite{Ito(CDC2002)}, an explicit construction method of the smooth Lyapunov function for the closed loop is proposed, with the proof given in Appendix \ref{app-sec-proof-Coro1}.
\begin{corollary}\label{Coro1}
Consider system (\ref{sys-x}) with parameter estimator (\ref{hat-theta}) and feedback controller (\ref{u-1}). Suppose that Assumptions \ref{ass-1}--\ref{ass-ideal-sys} hold, and (\ref{alpha-3}) is satisfied with $\tau_{err}>4$. Then the resulting closed-loop error system (\ref{tilde-theta-2}), (\ref{err-sys}) permits a smooth strict Lyapunov function having the sum-type form
\beeq{\label{V-cl-1}
V_{\rm cl}(t,e,\tilde\theta) = \int_{0}^{V_{err}}\lambda_{err}(s)ds + \int_{0}^{V_{est}}\lambda_{est}(s)ds\,
}
with $\lambda_{err},\lambda_{est}\in\mathcal{K}$. $\square$
\end{corollary}

\section{Robustness Analysis}
\label{sec-rob}

In this section, we demonstrate how to robustify the proposed adaptive controller by redesigning the feedback control law, such that  the resulting closed-loop system subject to input perturbation $d(t)$ is ISS. More explicitly, we consider perturbed nonlinear systems of the form
\beeq{\label{d-sys-x}
\dot x = f_1(x) + \phi(\theta,x) + g_1(x)(u + d(t))\,.
}

Instead of (\ref{u-1}), we design the robust  control law as
\beeq{\label{r-u}
u=\psi(e,x_r,\dot x_r,\mbox{satv}(\hat\theta + \beta(e,x_r)))-\eta(x_r,e)\,
}
where $\eta(\cdot)$ is the nonlinear damping function.

This, together with the estimator (\ref{hat-theta}), yields the resulting closed-loop error system with input perturbation $d(t)$ as
\beeq{\label{rclsys}\ba{rcl}
\dot e &=& F(e,x_r,\dot x_r,\mbox{satv}(\hat\theta + \beta(e,x_r))) -g(e+x_r)\eta(x_r,e) \,\\&& - \tilde\phi_s(\theta,\tilde\theta,e+x_r) + g(e+x_r)d(t)\\
\dot{\tilde\theta} &=& H(\theta,\tilde\theta,x_r) + \Delta(\theta,\tilde\theta,x_r,e) + \varsigma(e+x_r)g(e+x_r)d(t)
\ea}
Due to the presence of perturbation $d(t)$, two extra terms $g(e+x_r)d(t)$ and $\varsigma(e+x_r)g(e+x_r)d(t)$ appear in the $e$-subsystem and $\tilde\theta$-subsystem, respectively. With this in mind, we observe that there exist functions $\kappa_3,\kappa_4\in\mathcal{K}$ such that for all $e\in\mathbb{R}^n$ and $x_r\in\mathbb{B}_{r_1}^n$,
\beeq{\label{kappa-34}\ba{l}
|g(e+x_r)| \leq \kappa_3(|e|)+\kappa_3^\ast\,,\,\,
|\varsigma(e+x_r)| \leq \kappa_4(|e|) + \kappa_4^\ast
\ea}
holds with
$
\kappa_3^\ast=\sup\limits_{x_r\in\mathbb{B}_{r_1}^n}\|g(x_r)\|$, $ \kappa_4^\ast=\sup\limits_{x_r\in\mathbb{B}_{r_1}^n}\|\varsigma(x_r)\|$.

Let
\beeq{\label{nu}
\nu(s) = [\kappa_3(s)\kappa_4(s)+\kappa_4^\ast\kappa_3(s)+\kappa_3^\ast\kappa_4(s)]^2\,,
}
and
\beeq{\label{bar-kappa-2}
\bar\kappa_2(s) = l_{\gamma}\kappa_2(s)+\nu(s)\,.
}
\begin{theorem}\label{theo-rgeneral}
Suppose that Assumptions \ref{ass-1}--\ref{ass-ideal-sys} hold, and there exists a constant $\tau_{err}>1$ such that
\beeq{\label{r-alpha-3}
\alpha_3(s)\geq \tau_{err}\gamma_s\circ a^\ast \bar\kappa_2\circ\alpha_1^{-1}\circ\alpha_2(s)\cdot \alpha_4(s)\cdot\kappa_1(s)\,
}
holds for all $s\in\mathbb{R}_+$.
Choose the nonlinear damping term  as
\beeq{\label{r-eta}\ba{l}
\eta(x_r,e) = k_d\left[\dst\frac{\partial V_{err}}{\partial e}g(e+x_r)\right]^{\top}\,
\ea}
with $k_d>0$.
Then system (\ref{rclsys}) is uniformly ISS with respect to the states $(e,\tilde\theta)$ and input $d(t)$. $\square$
\end{theorem}

The proof of Theorem \ref{theo-rgeneral} is given in Appendix \ref{app-sec-proof-Theo-rg}.
Before the close of this section, it is worth noting that if (\ref{r-eta}) holds with $\tau_{err}>4$, then we can construct a smooth ISS Lyapunov function having the sum-type form as in (\ref{V-cl-1}). The explicit construction of such a smooth ISS Lyapunov function follows the proof of Corollary \ref{Coro1} and is thus omitted.

\begin{remark}
From (\ref{rclsys}), it can be seen that the perturbation $d(t)$ appears in both the $e$ and $\tilde\theta$ subsystems. To guarantee robust stability (i.e., ISS in this paper), it is natural to redesign the feedback control by introducing the nonlinear damping term $\eta(\cdot)$ in (\ref{r-eta}), which  can however only dominate the effect caused in the $e$ subsystem. As for the effect brought to the $\tilde\theta$-subsystem by $d(t)$, it can be seen that the corresponding ISS gain function is modified. This, as a consequence, requires a more restrictive condition (\ref{r-alpha-3}) by replacing $l_\gamma\kappa_2$ in (\ref{alpha-3}) by $\bar\kappa_2=l_{\gamma}\kappa_2+\nu$ in order to fulfill the ISS small-gain theorem. $\square$
\end{remark}

\begin{remark}
Despite this paper only considering the ISS robustness in the presence of the input perturbation, its extension to other kinds of perturbations such as parameter perturbation can be obtained by appropriately adapting the above arguments. $\square$
\end{remark}

\section{Removing the Need to Solve PDE (\ref{PDE})}
\label{sec-PDE}

In this section, we  present an approach to remove the need of solving the PDE (\ref{PDE}), which in the previous section, is required to derive the expression of $\beta(e,x_r)$.

We replace the function $\beta$ in (\ref{tilde-theta}) by
\beeq{\label{beta-a}
\beta_a(x_r,\hat e,e) = \varsigma(\hat e+x_r)e
}
with function $\varsigma$ satisfying Assumption \ref{ass-PE-2} and $\hat e$ being the state of a filter having the form
\beeq{\label{hat-e}\ba{rcl}
\dot {\hat e} = K(\epsilon_e) + f_1(x) + \phi(\mbox{satv}(\hat\theta+\beta_a),x)  + g_1(x)u- \dot x_r(t)
\ea}
where $\epsilon_e=e-\hat e$ and function $K(\cdot)$ is a design freedom.

We then design the parameter estimator as
\begin{dmath}\label{hat-theta-f}
\dot{\hat\theta} =  - \dst\varsigma(\hat e+x_r)\left[f_1(x) + \phi(\mbox{satv}(\hat\theta+\beta_a),x) + g_1(x)u - \dot x_r\right]
 -\dst\frac{\partial\beta_a}{\partial \hat e}\dot {\hat e}  - \dst\frac{\partial\beta_a}{\partial x_r}\dot {x}_r -k_{dz}\,\mbox{dzv}(\hat\theta+\beta_a)
\end{dmath}
and the feedback law $u$ as
\beeq{\label{u-f}
u=\psi(x_r,\dot x_r,\mbox{satv}(\hat\theta + \beta_a(e,x_r)),e)\,
}
where $\varsigma(\cdot)$ and $\psi(\cdot)$ are given in Assumptions \ref{ass-PE-2} and \ref{ass-ideal-sys}, respectively.

Thus, in the extended coordinates $(e,\tilde\theta,\epsilon_e)$, the resulting extended closed-loop system can be described by
\beeq{\label{clsys-f}\ba{l}
\dot e = F(x_r,\dot x_r,\mbox{satv}(\tilde\theta + \theta),e)  - \tilde\phi_s(\theta,\tilde\theta,e+x_r)\\
\dot{\tilde\theta} = H(\theta,\tilde\theta,x_r) + \Delta_f(\theta,\tilde\theta,x_r,e,\epsilon_e)\,\\
\dot \epsilon_e = -K(\epsilon_e)  - \tilde\phi_s(\theta,\tilde\theta,e+x_r)\,
\ea}
where
\[
\Delta_f = \varsigma(x_r)\tilde\phi_s(\theta,\tilde\theta,x_r)- \varsigma(e+x_r-\epsilon_e)\tilde\phi_s(\theta,\tilde\theta,e+x_r)
\]
satisfies $\Delta_f(\theta,0,x_r,e,\epsilon_e)=\Delta_f(\theta,\tilde\theta,x_r,0,0)=0$.

Instrumental to the subsequent analysis is the following property of functions $\tilde\phi_s$ and $\Delta_f$.
\begin{lemma}\label{lemma-varrho}
Suppose that Assumptions \ref{ass-1}, \ref{ass-2} hold. There exist functions $\varrho_i\in\mathcal{K}$, $i=1,2,3$ and a constant $\varrho_1^\ast\geq0$ such that
\begin{eqnarray}
&|\tilde\phi_s(\theta,\tilde\theta,e+x_r)| \leq (\varrho_1(|e|) + \varrho_1^\ast)\gamma_s(|\tilde\theta|)\,\label{t_phi_s}\\
&|\Delta_f(\theta,0,x_r,e,\epsilon_e)| \leq (\varrho_2(|e|)+\varrho_3(|\epsilon_e|))\gamma_s(|\tilde\theta|) \label{Delta_f_b}
\end{eqnarray}
for all $x_r(t)\in\mathbb{B}_{r_1}^n$, $e\in\mathbb{R}^n$. $\triangleleft$
\end{lemma}

With (\ref{Ineq-UP}), the proof of (\ref{t_phi_s}) in Lemma \ref{lemma-varrho} is straightforward by letting $\varrho_1(s)=\kappa_1(s)-\kappa_1(0)$ and $\varrho_1^\ast=\kappa_1(0)$, while
the proof of (\ref{Delta_f_b}) is similar to that of Lemma \ref{lemma-varphi} and is thus omitted.
Following Lemma \ref{lemma-non-ex} and Theorem \ref{theo-general}, both $e$ and $\tilde\theta$ subsystems are uniformly ISS with an appropriate choice of $\psi(\cdot)$. On the other hand, for the $\epsilon_e$ subsystem, the stabilizing term $K(\epsilon_e)$ can always be chosen such that the $\epsilon_e$-subsystem is also uniformly ISS with respect to state $\epsilon_e$ and inputs $e,\tilde\theta$. In this way, the extended system (\ref{clsys-f}) is a feedback interconnection of three uniform ISS subsystems, for which Theorem \ref{theo-csgt} can be employed to show  uniform global asymptotic stability. Motivated by these observations, in what follows a sufficient condition on $\psi(\cdot)$ and $K(\cdot)$ is presented to achieve the uniform global asymptotic stability of the extended system  (\ref{clsys-f}), with the proof given in Appendix \ref{app-proof-theo-f}.

\begin{theorem}\label{theo-f}
Consider system (\ref{sys-x}) with filter (\ref{hat-e}), parameter estimator (\ref{hat-theta-f}) and feedback controller (\ref{u-1}). Suppose that Assumptions \ref{ass-PE-2}, \ref{ass-ideal-sys} hold, and there exists a constant $\tau_{err}>1$ such that
\beeq{\label{alpha-3-f}
\alpha_3(s)\geq \tau_{err}\gamma_s\circ a^\ast l_{\gamma}\varrho_2\circ\alpha_1^{-1}\circ\alpha_2(s)\cdot \alpha_4(s)\cdot\kappa_1(s)\,
}
for all $s\in\mathbb{R}_+$. Suppose there exist functions $K_i(\cdot)$, $i=1,2$ such that
\beeq{\ba{l}
\epsilon_e^{\top}K_1(\epsilon_e) \geq |\epsilon_e|\cdot\pi_{\epsilon,e}\circ\check\gamma_{e,\tilde\theta}\circ\check\gamma_{\tilde\theta,\epsilon}(|\epsilon_e|^2)\,\\
\epsilon_e^{\top}K_2(\epsilon_e) \geq |\epsilon_e|\cdot\pi_{\epsilon,\tilde\theta}\circ\check\gamma_{\tilde\theta,\epsilon}(|\epsilon_e|^2)
\ea}
where
\begin{eqnarray}
\check\gamma_{\tilde\theta,e}(s) &=& a_1(\tau_{est}a^\ast l_\gamma)^2[\varrho_2\circ\alpha_1^{-1}(s)]^2\,\\
\check\gamma_{\tilde\theta,\epsilon}(s) &=& \dfrac{4\tau_{est}^2}{(\tau_{est}-1)^2}a_1(\tau_{est}a^\ast l_\gamma)^2[\varrho_3(\sqrt{s})]^2\,\\
\check\gamma_{e,\tilde\theta}(s) &=& \check\gamma_{\tilde\theta,e}^{-1}(\frac{\tau_{est}}{\tau_{err}'}s)\,\\
\pi_{\epsilon,\tilde\theta}(s) &=& \varrho_1^\ast\gamma_s(\sqrt{\frac{s}{a_1}}) +\frac{1}{4}\gamma_s(\sqrt{\frac{s}{a_1}})^2\, \label{pi-theta}\\
\pi_{\epsilon,e}(s) &=& [\varrho_1\circ\alpha_1^{-1}(s)]^2
\end{eqnarray}
with $\tau_{err}>\tau_{err}'>\tau_{est}>1$.
Then choosing
\beeq{
K(\epsilon_e) = k_\epsilon \epsilon_e + K_1(\epsilon_e) + K_2(\epsilon_e)
}
with $k_\epsilon>0$, the zero equilibrium point of the extended system (\ref{clsys-f}) is uniformly globally asymptotically stable.  $\square$
\end{theorem}

\begin{remark}
As a particular case, (\ref{beta-a}) can be replaced by $\beta_a(x_r,e) = \varsigma(x_r)e$. In this case, the filter (\ref{hat-e}) is not required. However, since the variable $e$ is not included in the function $\varsigma$, the resulting $\kappa_2$ in Lemma \ref{lemma-1} cannot be shaped by the design freedom $\varsigma$, making the corresponding condition (\ref{alpha-3}) or (\ref{r-alpha-3}) more restrictive. $\square$
\end{remark}

\section{Adaptive Tracking of Series Elastic Actuators}
\label{sec-SEAs}

In this section, we demonstrate how to use the proposed adaptive control scheme to deal with the tracking problem for series elastic actuators (SEAs), which can be described by the following differential equations:
\beeq{\label{sys-SEAs}\ba{rcl}
m \ddot d &=& - F_s(d) - \mu_v \dot d + c_f i\,\\
L\dot i &=& - R i - c_b \dot d + V_{in}
\ea}
where $d$ and $i$ denote the spring deflection and the armature current, respectively. The control input $V_{in}$ is the armature voltage, and the function $F_s(d)$ denotes the elastic force of the nonlinear spring \cite{Stulov(2003),Boisseau(2012)}, which can be approximately modelled by a power law of the form
$
F_s(d)=Q_0 d|d|^p
$
with unknown positive constants $Q_0$ and $p$, taking values in some known compact sets, i.e., $0< Q_{0,l}\leq Q_0\leq Q_{0,u}$ and $0<\underline p \leq p \leq\bar p$.
The quantity $m$ is the mass of the moving parts, $\mu_v$ is the viscous friction constant, $c_f$ is the force constant, $c_b$ is the back-electromotive-force constant, and $L$ and $R$ are respectively, the inductance and the resistance of the armature. The control problem is to adjust the DC motor in order to drive the moving end of the spring to follow a trajectory $d_r(t)=\mbox{exp}(\sin t)$, i.e., $\lim\limits_{t\rightarrow\infty}|d(t)-d_r(t)|=0$.

In this setting, let $x_1=d$, $x_2=\dot d$, $x_3=\frac{c_f}{m}i$, $u=\frac{c_f}{mL} V_{in}- \frac{c_fR}{mL} x_3 - \frac{c_fc_b}{mL} x_2$, which transforms (\ref{sys-SEAs}) into the form of (\ref{sys-x}) as
\beeq{\label{sys-x-SEA}\ba{rcl}
\dot x_1 &=& x_2\,\\
\dot x_2 &=& -\phi(\theta,x_1) - b_3 x_2 +  x_3\,\\
\dot x_3 &=&  u
\ea}
where $\phi(\theta,x_1)$ is a re-parameterized function of the form
\[
\dst\phi(\theta,x_1) = b_1x_1 \mbox{exp}(b_2\theta_1) |x_1|^{\theta_2+p^\ast}
\]
with $p^\ast=0.5(\bar p+\underline p)$,
\[\ba{l}
b_1=\frac{Q_{0,l}+Q_{0,u}}{2m}\,,\\
b_2 = \frac{2}{\underline p}\max\left\{\log\frac{2Q_{0,u}}{Q_{0,u}+Q_{0,l}}, \log\frac{Q_{0,u}+Q_{0,l}}{2Q_{0,l}}\right\}\,\\
b_3= \frac{\mu_v}{m}
\ea\]
and the unknown parameter vector
\[
\theta=\begin{pmatrix}\theta_1 \cr \theta_2\end{pmatrix}:=\begin{pmatrix}\dfrac{1}{b_2}\log\left(\dfrac{2Q_0}{ Q_{0,l}+Q_{0,u}}\right)\cr p-0.5(\underline p+ \bar p) \end{pmatrix}\,.
\]
It is clear that $\theta_1\in[-0.5 \underline p,0.5\underline p]$ and $\theta_2\in[-0.5 (\bar p-\underline p), 0.5 (\bar p- \underline p)]$, leading to $|\theta|\leq l_\theta := 0.5 \bar p$.

In this setting, we now proceed to deal with the adaptive tracking problem of the nonlinearly parameterized system (\ref{sys-x-SEA}). It is observed that the uncertain parameter $\theta$ appears in the equation of $\dot e_2$, rather than that of the control $u$, which means that the matching condition is not satisfied. To overcome this obstacle, the backstepping technique will be employed. More specifically, the whole design will be divided into three steps as below.

{\bf Step 1:} Defining $e_1=x_1-d_r(t)$ and $e_2=x_2 - \tau_2$, we rewrite
\beeq{\label{e1-SEA}
\dot e_1 = e_2 + \tau_2 -\dot d_r\,.
}
By choosing $\tau_2= -(k_1+0.5) e_1+\dot d_r$ with $k_1>1$ and the Lyapunov function $V_1(e_1)=|e_1|^2$, we have
\[
\dot V_1 = - (2k_1+1) |e_1|^2 + 2 e_1 e_2\,.
\]
This, by setting $V_2(e_2)=|e_2|^2$ and
\beeq{\label{gamma-12}
\gamma_{1,2}(s)=\frac{1}{k_1^2}s\,,\quad s\in\mathbb{R}_+\,,
}
implies
\beeq{
V_1 \geq \gamma_{1,2}(V_2) \,\Longrightarrow\, \dot V_1 \leq -  V_1\,.
}

{\bf Step 2:}
We now proceed to the second equation of (\ref{sys-x-SEA}) by viewing $x_3$ as the control variable. Defining
$e_3 = x_3 - \tau_3$, we compute the derivative of $e_2$ as
\beeq{\label{e2-MEAs-1}
\dot e_2 = e_3 + \tau_3 - \phi(\theta,e_1+d_r)  + f_2
}
where
\[
f_2=\left(k_1+0.5-b_3\right)(e_2+\tau_2) -(k_1+0.5)\dot d_r - \ddot d_r\,.
\]

Following the proposed design method, we  design the parameter estimator as
\beeq{\ba{l}
\dot{\hat\theta} = -\varsigma(d_r+e_1)[e_3 + \tau_3-\phi(\mbox{satv}(\hat\theta+\beta),e_1+d_r) + f_2]\,\\ \qquad - \frac{\partial\beta}{\partial e_1}\dot e_1 - \frac{\partial\beta}{\partial d_r}\dot d_r - k_{dz}\mbox{dzv}(\hat\theta+\beta)\,,
\ea}
where $\mbox{satv}(s_1,s_2)=\col(\mbox{sat}(s_1),\mbox{sat}(s_2))$ with $\mbox{sat}(\cdot)$  a smooth saturation function with saturation level $l_s> l_\theta$ and $|\mbox{sat}(s)|< p^\ast$ for all $s\in\mathbb{R}$. We choose $\beta(d_r,e_1,e_2) = \varsigma(d_r+e_1)e_2$ and
\[
\varsigma(d_r+e_1)= -(e_1+d_r)\begin{pmatrix} b_2 \cr \log|e_1+d_r|\end{pmatrix}\,.
\]
Thus, by setting $\tilde\theta = \hat\theta - \theta + \beta$, we obtain
\beeq{\label{estsys-SEA}
\dot {\tilde\theta} = -\varsigma(d_r+e_1)\tilde\phi_{s}(\theta,\tilde\theta,e_1+d_r)- k_{dz}\mbox{dzv}(\tilde\theta+\theta)
}
where
\[
\tilde\phi_{s}(\theta,\tilde\theta,e_1+d_r)=\phi(\theta,e_1+d_r)-\phi(\mbox{satv}(\tilde\theta+\theta),e_1+d_r)\,.
\]
\begin{lemma}\label{lemma-SEA}
There exists a $k_{dz}^\ast>0$ such that for all $k_{dz}> k_{dz}^\ast$, system (\ref{estsys-SEA}) admits a uniform ISS Lyapunov function $V_{\tilde\theta}(t,\tilde\theta)$ such that
\beeq{\label{SEA-V-theta}\ba{l}
a_1|\tilde\theta|^2 \leq V_{\tilde\theta}(t,\tilde\theta) \leq a_2|\tilde\theta|^2\,\\
V_{\tilde\theta} \geq \gamma_{\tilde\theta,1}(V_1) \,\Longrightarrow \, \dot V_{\tilde\theta}(t,\tilde\theta) \leq -a_3V_{\tilde\theta}\,
\ea}
for some constants $a_i>0$, $i=1,\ldots,3$ and a class $\mathcal{K}_{\infty}$ function $\gamma_{\tilde\theta,1}$, satisfying $\gamma_{\tilde\theta,1}(s)\in\mathcal{O}(s)$ as $s\rightarrow 0^+$. $\square$
\end{lemma}

The proof of Lemma \ref{lemma-SEA} is given in Appendix \ref{app-sec-proof-Pro1}. Note that since $\gamma_{\tilde\theta,1}(s)\leq \bar \gamma_{\tilde\theta,1}(s) s$ for some $\bar\gamma_{\tilde\theta,1}\in\mathcal{SN}$ and all $s\in\mathbb{R}_+$, and $\gamma_s$ satisfies (\ref{L-gamma}) with $l_\gamma=2p^\ast$, it can be seen that for $s> 1$,
\[
\gamma_s\left(\sqrt{{\gamma_{\tilde\theta,1}(s)}/{a_1}}\right)^2 \leq l_\gamma^2 s
\]
and for $0\leq s\leq 1$,
\[
\gamma_s\left(\sqrt{{\gamma_{\tilde\theta,1}(s)}/{a_1}}\right)^2 \leq {\gamma_{\tilde\theta,1}(s)}/{a_1} \leq g_1 s
\]
with $g_1=\max_{0\leq s\leq 1}\frac{\bar\gamma_{\tilde\theta,1}(s)}{a_1}$.
Thus, there exists a constant $g>0$ such that
\[
\gamma_s\left(\sqrt{{\gamma_{\tilde\theta,1}(s)}/{a_1}}\right)^2 \leq  g^2\,s \,,\quad \forall s\in\mathbb{R}_+\,.
\]

With this being the case, we turn to consider (\ref{e2-MEAs-1}) and choose
\beeq{\label{tau2-MEA}\ba{l}
\tau_3=-k_{21} e_2 - k_{22}e_2|e_2|^{2p^\ast+1}-f_2 \,\\ \qquad + \phi(\mbox{satv}(\hat\theta+\beta),e_1+d_r)\,.
\ea}
This leads to
\beeq{\label{e2-SEAs-2}\ba{l}
\dot e_2 = e_3-k_{21} e_2 - k_{22}e_2|e_2|^{2p^\ast+1} \,\\ \qquad -\tilde\phi_{s}(\theta,\tilde\theta,e_1+d_r)
\ea}
where by some simple but lengthy calculations, the last term satisfies
\[
|\tilde\phi_{s}(\theta,\tilde\theta,e_1+d_r)| \leq (\delta_1 |e_1|^{2p^\ast+2} + \delta_2)\gamma_s(|\tilde\theta|)
\]
for some constants $\delta_i>0$, $i=1,2$.

Thus, computing the derivative of $V_2$ along (\ref{e2-SEAs-2}) yields
\[\ba{l}
\dot V_2 \leq -2|e_2|( k_{21}|e_2| - |e_3|- \delta_2\gamma_s(|\tilde\theta|)) \, \\ \qquad -  2|e_2|(k_{22}|e_2|^{2p^\ast+2}-\delta_1 l_\gamma|e_1|^{2p^\ast+2}) \,\\
\ea\]
which in turn implies that system (\ref{e2-SEAs-2}) is uniformly ISS with respect to inputs $e_1,\tilde\theta,e_3$, with an ISS Lyapunov function $V_{2}(e_2)=|e_2|^2$, satisfying
\beeq{\label{V2-SEAs}\ba{l}
V_{2} \geq \mbox\{\gamma_{2,1}(V_1),\gamma_{2,\tilde\theta}(V_{\tilde\theta}),\gamma_{2,3}(V_3)\} \,\Longrightarrow \, \dot V_{2} \leq -V_{2}
\ea}
with $V_3=|e_3|^2$, $k_{21}>g\delta_2 + 2.5$ and $k_{22}\geq \delta_1l_\gamma$, and
\beeq{\ba{l}
\gamma_{2,1}(s) = s\,\\
\gamma_{2,\tilde\theta}(s) = \dfrac{\delta_2^2}{(k_{21}-2.5)^2}\gamma_s\left(\sqrt{{s}/{a_1}}\right)^2\,\\
\gamma_{2,3}(s) = 0.5s\,.
\ea}
Importantly, with the above construction we have
\beeq{\label{SEA-csmc-1}\ba{l}
\gamma_{1,2}\circ\gamma_{2,1} < \mbox{Id} \,\\
\gamma_{1,2}\circ\gamma_{2,\tilde\theta}\circ\gamma_{\tilde\theta,1} < \mbox{Id}\,.
\ea
}

{\bf Step 3:} At this final step, the actual control law $u$ will be designed. Computing the time derivative of $e_3$ yields
\beeq{\label{e3-SEAs}\ba{l}
\dot e_3 = u - \dfrac{\partial\tau_3}{\partial t}- \dfrac{\partial\tau_3}{\partial e_1}\dot e_1 - \dfrac{\partial\tau_3}{\partial e_2}\dot e_2 - \dfrac{\partial\tau_3}{\partial \hat\theta}\dot {\hat\theta}\,.
\ea}
Choosing
\beeq{\label{u-MEAs}\ba{rcl}
u &=& \bar u + \dfrac{\partial\tau_3}{\partial t}+ \dfrac{\partial\tau_3}{\partial e_1}\dot e_1  + \dfrac{\partial\tau_3}{\partial \hat\theta}\dot {\hat\theta} \,\\ && + \dfrac{\partial\tau_3}{\partial e_2}(e_3-k_{21} e_2 - k_{22}e_2|e_2|^{4p^\ast+4})
\ea}
with $\bar u$ being the residual control to be determined,
we compute the derivative of the Lyapunov function $V_3=|e_3|^2$ as
\[\ba{l}
\dot V_3 = 2e_3\bar u + 2e_3\dfrac{\partial\tau_2}{\partial e_2}\tilde\phi_{s}(\theta,\tilde\theta,e_1+d_r)
\ea\]
where
\[\ba{l}
\dfrac{\partial\tau_2}{\partial e_2} = - k_{21} - k_{22}(2p^\ast+2)e_2^{2p^\ast+1} - (k_1+0.5-b_3)e_2\,\\  + \phi(\mbox{satv}(\hat\theta+\beta),e_1+d_r) \begin{pmatrix} b_5 \cr \log|e_1+d_r| \end{pmatrix}^{\top} \dfrac{\partial\mbox{satv}}{\partial \beta}\varsigma(e_1+d_r)\,.
\ea\]
Some simple but lengthy computations then show that
\[
\left|\dfrac{\partial\tau_2}{\partial e_2}\tilde\phi_{s}(\theta,\tilde\theta,e_1+d_r)\right| \leq \rho_0\gamma_s(|\tilde\theta|)+\rho_1|e_1|^{4p^\ast+6} + \rho_2|e_2|^{4p^\ast+2}
\]
holds for some constants $\rho_i>0$.

Thus, choosing
\beeq{\label{baru}
\bar u = -(k_{31} + k_{32}|e_3|^{4p^\ast+5} + k_{33}|e_3|^{4p^\ast+1})e_3\,,
}
it can be deduced that
\beeq{\label{V3-SEAs}\ba{l}
V_{3} \geq \mbox\{\gamma_{3,1}(V_1),\gamma_{3,\tilde\theta}(V_{\tilde\theta}),\gamma_{32}(V_2)\} \Longrightarrow  \dot V_{3}(e_3) \leq -V_{3}
\ea}
with $k_{31}>g\rho_0  + 1.5$, $k_{32}\geq \rho_1$, $k_{33}\geq\rho_2$ and
\beeq{\ba{l}
\gamma_{3,1}(s) = s\,\\
\gamma_{3,\tilde\theta}(s) = \dfrac{\rho_0^2}{(k_{31}-1.5)^2}\gamma_s\left(\sqrt{{s}/{a_1}}\right)^2\,\\
\gamma_{3,2}(s) = s\,.
\ea}
With the above construction we have
\beeq{\label{SEA-csmc-2}\ba{l}
\gamma_{1,2}\circ\gamma_{2,3}\circ\gamma_{3,1} < \mbox{Id} \,\\
\gamma_{2,3}\circ\gamma_{3,2} < \mbox{Id}\,\\
\gamma_{1,2}\circ\gamma_{2,3}\circ\gamma_{3,\tilde\theta}\circ\gamma_{\tilde\theta,1} < \mbox{Id}\,.
\ea
}
We observe that the resulting closed-loop system can be viewed as a networked system  consisting of 4 ISS subsystems: $e_1$ subsystem (\ref{e1-SEA}), $e_2$-subsystem (\ref{e2-SEAs-2}), $\tilde\theta$-subsystem (\ref{estsys-SEA}) and $e_3$-subsystem (\ref{e3-SEAs}). Moreover, this network is comprised of 5 simple cycles, for which the cyclic small-gain conditions (\ref{csmc}) are verified to be true by  (\ref{SEA-csmc-1}) and (\ref{SEA-csmc-2}). Therefore, according to Theorem \ref{theo-csgt}, the uniform global asymptotic stability for the resulting closed-loop system can be easily summarized as below.
\begin{proposition}
Consider the closed-loop system (\ref{e1-SEA}), (\ref{estsys-SEA}),(\ref{e2-SEAs-2}), (\ref{e3-SEAs}), (\ref{u-MEAs}), and (\ref{baru}).
Then the zero equilibrium point of the resulting closed-loop system is uniformly globally asymptotically stable. $\square$
\end{proposition}

To verify the validity of the proposed controller, the simulation is performed with unknown parameters $\theta=\col(0.2,0.4)$ and design parameters $k_{21}=2$, $k_{21}=5$, $k_{22}=10$, $k_{31}=50$, $k_{32}=100$, $k_{33}=100$. As seen from Figures \ref{fig_SEA-e1}, \ref{fig_SEA_theta1} and \ref{fig_SEA_theta22}, the resulting trajectories of the tracking error $e_1(t)$ and the parameter estimation errors $\tilde\theta=\col(\tilde\theta_1(t),\tilde\theta_2(t))$ asymptotically converge to zero.
\begin{figure}[thpb]
\begin{center}
\centering\includegraphics[height=50mm,width=70mm]{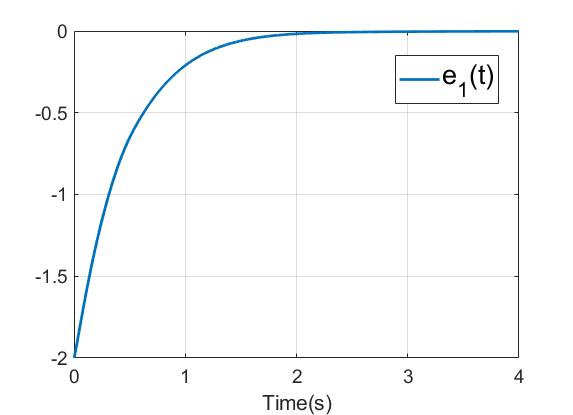} \caption{Trajectory of the tracking error $e_1(t)$}
\label{fig_SEA-e1}
\end{center}
\end{figure}

\begin{figure}[thpb]
\begin{center}
\centering\includegraphics[height=50mm,width=70mm]{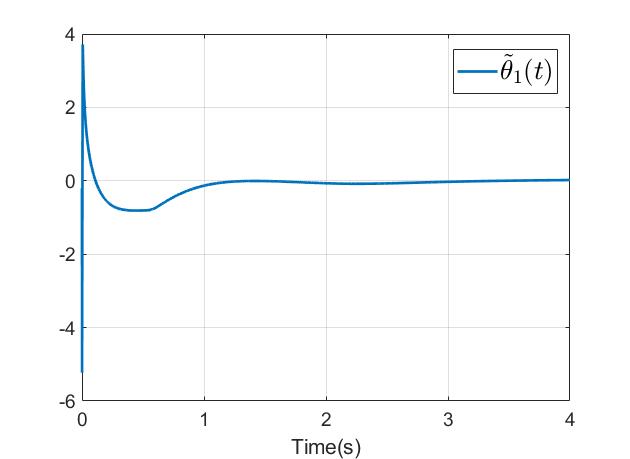} \caption{Trajectory of the parameter estimation error $\tilde\theta_1(t)$}
\label{fig_SEA_theta1}
\end{center}
\end{figure}

\begin{figure}[thpb]
\begin{center}
\centering\includegraphics[height=50mm,width=70mm]{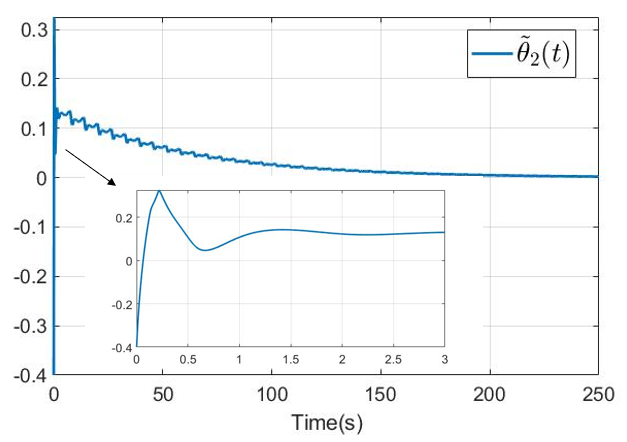} \caption{Trajectory of the parameter estimation error $\tilde\theta_2(t)$}
\label{fig_SEA_theta22}
\end{center}
\end{figure}

\section{Conclusions}
\label{sec-con}
This paper investigates the robust I\&I adaptive tracking problem for a class of nonlinearly parameterized systems from the perspective of ISS. Compared to the standard I\&I adaptive method, a saturation function and a deadzone function are introduced in such a way that an interconnection of an ISS estimation error system and an ISS tracking error subsystem is derived under some mild assumptions. According to an ISS small-gain condition, the desired uniform global asymptotic stability of the resulting interconnected ``error" system can be achieved and a sum-type strict Lyapunov function can be explicitly constructed. Taking advantage of this ISS-based design framework, it is shown that the corresponding robustness with respect to the input perturbation can be rendered to be ISS. To remove the need of solving the immersion manifold shaping PDE, a new filter-based approach is proposed, which preserves the ISS-based design framework. In terms of future works, it is worth considering its applicability to deal with global adaptive nonlinear output regulation problem \cite{WangKellett2019,Wang&Chen&Xu}. Another interesting topic is to relax the PE condition \cite{EfimovBarabanov&Ortega(cdc18)}.


\appendix

\section{Proof of Lemma \ref{lemma-varphi}}
\label{app-varphi}

\vspace{-1.0em}
{
Observe that the function $\varphi(\theta',x_r,e+x_r)$ is continuously differential with respect to $\theta'\in\mathbb{R}^q$ by Assumption \ref{ass-1}, and $\varphi(\theta',x_r,x_r)=0$ for all $\theta'\in\mathbb{R}^q$ and $x_r\in\mathbb{B}_{r_1}^n$, with $r_1$ given in Assumption \ref{ass-2}.

Let $\varphi_i(\cdot)$ denote the $i$-th entry of $\varphi(\cdot)$ and define $\hat{\varphi}_i(j):=\varphi_i(\theta^\dag+j(\theta^\ddag-\theta^\dag),x_r,e+x_r)$ for $j\in\mathbb{R}$. It is clear that $\hat{\varphi}_i(0)=\varphi_i(\theta^\dag,x_r,e+x_r)$ and $\hat{\varphi}_i(1)=\varphi_i(\theta^\ddag,x_r,e+x_r)$.
Thus, applying the Mean Value Theorem yields that given any $\theta^\dag,\theta^\ddag\in\mathbb{B}^q_{l_s+0.5\epsilon_s}$, there exists $j_i'\in [0,1]$ such that the equality
\[\ba{l}
\varphi_i(\theta_1,x_r,e+x_r)-\varphi_i(\theta_2,x_r,e+x_r) = \hat\varphi_i(0)-\hat\varphi_i(1)\,\\
 =-\dst\frac{\partial \hat\varphi_i(j_i^\prime)}{\partial j_i^\prime}= \dst\frac{\partial \varphi_i}{\partial \theta_i^\prime}(\theta_i^\prime,x_r,e+x_r)(\theta^\dag-\theta^\ddag)
\ea\]

\vspace{-1.6em}\noindent
holds with $\theta_i^\prime := \theta^\dag+j_i'(\theta^\ddag-\theta^\dag)\in\mathbb{B}^q_{l_s+0.5\epsilon_s}$,  where the function $\dst\frac{\partial \varphi_i}{\partial \theta'}(\theta_i',x_r,e+x_r)$ is continuous in $e\in\mathbb{R}^n,\theta_i'\in\mathbb{R}^q$ and $x_r\in\mathbb{B}_{r_1}^n$. Since $\varphi_i(\theta_i^\prime,x_r,x_r)\equiv0$ for all $\theta_i^\prime\in\mathbb{R}^q$ by definition, we have
\[
\dst\frac{\partial \varphi_i}{\partial \theta'}(\theta_i',x_r,x_r) = 0 \quad \mbox{for all $\theta_i'\in\mathbb{R}^q$ and $x_r\in\mathbb{B}_{r_1}^n$}\,.
\]

\vspace{-1.6em}\noindent
With this in mind, let
\[
\kappa_2(s) = \displaystyle\sum_{i=1}^q\sup_{(\theta_i',x_r,e)\in\mathbb{B}^q_{l_s+0.5\epsilon_s}\times \mathbb{B}_{r_1}^n\times\mathbb{B}_s^n}\left|\dst\frac{\partial \varphi_i}{\partial \theta_i'}(\theta_i',x_r,e+x_r)\right|
\]
for $s\in\mathbb{R}_+$. Clearly, $\kappa_2(s)$ is a continuous non-decreasing function, and $\kappa_2(0)=0$.
This implies $\kappa_2\in\mathcal{K}$. Hence,
we have
\[
|\varphi(\theta^\dag,x_r,e+x_r)-\varphi(\theta^\ddag,x_r,e+x_r)| \leq  \kappa_2(|e|)|\theta^\dag-\theta^\ddag|\,.
\]
This completes the proof. $\blacksquare$
}

\section{Proof of Lemma \ref{lemma-1}}
\label{app-1}

\vspace{-1.0em}
It is observed that for all $\theta\in\Theta$, if $|\theta'|\leq l_s$, then
\[\ba{rcl}
&&(\theta'-\theta)^{\top}H(\theta,\theta'-\theta,x_r) \,\\
&\leq& -(\theta'-\theta)^{\top}M_1(x_r)(\theta'-\theta) - k_{dz}(\theta'-\theta)^{\top} \mbox{dzv}(\theta')\,\\
&\leq& -(\theta'-\theta)^{\top}M_1(x_r)(\theta'-\theta)\,,
\ea\]
where the first inequality is obtained by using (\ref{ineq-ND}), and the second is obtained by using the fact that $(\theta'-\theta)^{\top} \mbox{dzv}(\theta')\geq 0$ for all $\theta\in\Theta$ and $\theta'\in\mathbb{R}^q$.

If $|\theta'|> l_s$, simple calculations show that there exists  an $r_3>0$ such that
\[
r_3(\theta'-\theta)^{\top} \mbox{dzv}(\theta')\geq |\theta'-\theta|^2\,.
\]
Due to the presence of the saturation function $\mbox{satv}(\cdot)$ in the definition of $\tilde\phi_s$ in (\ref{def-tphi})
, there exists an $r_4>0$ such that
\[
|\varsigma(x_r)\tilde\phi_s(\theta,\theta'-\theta,x_r)| \leq r_4
\]
for all $\theta\in\Theta$ and $x_r\in\mathbb{B}_{r_1}^n$, with $r_1$ given in Assumption \ref{ass-2}.
Thus, for all $|\theta'|> l_s$, we have
\[\ba{rcl}
&&(\theta'-\theta)^{\top}H(\theta,\theta'-\theta,x_r)\,\\
&\leq& |\theta'-\theta||\varsigma(x_r)\tilde\phi_s(\theta,\theta'-\theta,x_r)|- k_{dz} (\theta'-\theta)^{\top}\mbox{dzv}(\theta')\,\\
&\leq& r_4|\theta'-\theta| - \dst\frac{k_{dz}}{r_3}|\theta'-\theta|^2\,\\
 &\leq& -\dst\frac{(k_{dz}-k_{dz}^\ast)}{r_3}|\theta'-\theta|^2 - (\theta'-\theta)^{\top}M_1(x_r)(\theta'-\theta)\,
\ea\]

\vspace{-1.6em}\noindent
holds with
$
k_{dz}^\ast :=\dst\frac{r_3r_4}{2 l_\theta  +1} + r_3\max_{x_r\in\mathbb{B}_{r_1}^n}\|M_1(x_r)\|\,.
$

\vspace{-1.0em}
Therefore, by choosing $k_{dz}\geq k_{dz}^\ast$, we have
\[
(\theta'-\theta)^{\top}H(\theta,\theta'-\theta,x_r) \leq - (\theta'-\theta)^{\top}M_1(x_r)(\theta'-\theta)\,,
\]
which, by using $\tilde\theta$ to replace $\theta'-\theta$, yields that for all $\tilde\theta\in\mathbb{R}^q$,
\beeq{\label{A-1}
\tilde\theta^{\top}H(\theta,\tilde\theta,x_r) \leq - \tilde\theta^{\top}M_1(x_r)\tilde\theta \leq 0\,.
}

\vspace{-1.6em}\noindent
With this in mind, let $\Phi(t,t_0,\tilde\theta)$ denote the solution of system (\ref{auxi-tilde-theta-2}) that starts at $\tilde\theta$. It is clear that
\beeq{
{\partial \Phi(\tau,t,\tilde\theta)}/{\partial \tau} = H(\theta,\Phi(\tau,t,\tilde\theta),x_r(\tau))\,.
}
By (\ref{A-1}), it is observed that
\[\ba{rcl}
{\partial |\Phi(\tau,t,\tilde\theta)|^2}/{\partial \tau}
 &\leq& -2\Phi(\tau,t,\tilde\theta)^{\top}M_1(x_r(\tau))\Phi(\tau,t,\tilde\theta)\,\\
 &\leq& 0\,.
\ea\]

\vspace{-1.6em}\noindent
This indicates that for any $\tau\geq t$, $|\Phi(\tau,t,\tilde\theta)|^2 \leq |\Phi(t,t,\tilde\theta)|^2$, and
\beeq{\label{bound-W}
|\tilde\theta(t)|^2 \leq \tilde\theta(t_0)^{\top}\mbox{exp}\left(-\int_{t_0}^{t}  \mathcal{H}(M_1(x_r(\tau))) d\tau\right)\tilde\theta(t_0)
}

\vspace{-1.6em}\noindent
where $\mathcal{H}(M_1(x_r(\tau)))=M_1(x_r(\tau))^{\top}+M_1(x_r(\tau))$, with $t_0$ being the initial time.

\vspace{-1.0em}
Then, with $\delta>0$ we choose the Lyapunov function
\beeq{\label{V-1}
V_{est}(t,\tilde\theta) = \int_{t}^{t+\delta} |\Phi(\tau,t,\tilde\theta)|^2 d\tau\,.
}

\vspace{-1.6em}\noindent
It is immediate to see that
\beeq{\label{low-V-1}
 V_{est}(t,\tilde\theta) \leq \int_{t}^{t+\delta} |\Phi(t,t,\tilde\theta)|^2 d\tau = \delta |\tilde\theta|^2\,.
}
On the other hand, recalling  (\ref{Ineq-UP}) and the definition of function $\mbox{dz}(\cdot)$ in (\ref{def-dz}),  there exist constants $c_1,c_2>0$ such that
\[\ba{l}
|\Phi(\tau,t,\tilde\theta)||\varsigma(x_r(\tau))||\tilde\phi_s(\theta,\Phi(\tau,t,\tilde\theta),x_r(\tau))| \leq c_1|\Phi(\tau,t,\tilde\theta)|^2\,\\
|k_{dz}\Phi(\tau,t,\tilde\theta)\mbox{dzv}(\Phi(\tau,t,\tilde\theta)+\theta)| \leq c_2|\Phi(\tau,t,\tilde\theta)|^2\,.
\ea\]

\vspace{-1.8em}
\noindent
This in turn indicates that
\vspace{-1.5em}
\begin{dmath*}
{\partial |\Phi(\tau,t,\tilde\theta)|^2}/{\partial \tau}
 = 2\Phi(\tau,t,\tilde\theta)^{\top}H(\theta,\Phi(\tau,t,\tilde\theta),x_r(\tau))\,\\
 \geq -2|\Phi(\tau,t,\tilde\theta)||\varsigma(x_r(\tau))||\tilde\phi_s(\theta,\Phi(\tau,t,\tilde\theta),x_r(\tau))| \\ \qquad - k_{dz}|\Phi(\tau,t,\tilde\theta)\mbox{dzv}(\Phi(\tau,t,\tilde\theta)+\theta)|\,\\
\geq -2(c_1+c_2) |\Phi(\tau,t,\tilde\theta)|^2\,.
\end{dmath*}

\vspace{-1.5em}
By standard arguments, it then follows that
$
|\Phi(\tau,t,\tilde\theta)|^2 \geq |\tilde\theta|^2e^{-2(c_1+c_2)(\tau-t)}\,,
$
which leads to
\beeq{\label{up-V-1}
V_{est}(t,\tilde\theta) \geq \frac{1-e^{-2(c_1+c_2)\delta}}{2(c_1+c_2)}|\tilde\theta|^2\,.
}

\vspace{-1.5em}
\noindent This, together with (\ref{low-V-1}) proves (\ref{Vest-a}) with $a_2=\delta$ and $a_1=\frac{1-e^{-2(c_1+c_2)\delta}}{2(c_1+c_2)}$.

\vspace{-1.0em}
We now proceed to compute the time derivative of $V_{est}$ as
\[\ba{l}
\dot V_{est} = \dst\frac{\partial V_{est}(t,\tilde\theta)}{\partial t} + \dst\frac{\partial V_{est}(t,\tilde\theta)}{\partial \tilde\theta}H(\theta,\tilde\theta,x_r(t))\,\\
 = |\Phi(t+\delta,t,\tilde\theta(t))|^2 -|\tilde\theta(t)|^2 \\  + 2 \dst\int_{t}^{t+\delta} \Phi(\tau,t,\tilde\theta)^{\top}  \left[\frac{\partial\Phi(\tau,t,\tilde\theta)}{\partial t}+\frac{\partial \Phi(\tau,t,\tilde\theta)}{\partial \tilde\theta}H(\theta,\tilde\theta,x_r)\right]d\tau\,\\
 = |\Phi(t+\delta,t,\tilde\theta(t))|^2 -|\tilde\theta(t)|^2
\ea\]
where to obtain the last equation we have used the fact that
\[
\frac{\partial\Phi(\tau,t,\tilde\theta)}{\partial t}+\frac{\partial \Phi(\tau,t,\tilde\theta)}{\partial \tilde\theta}H(\theta,\tilde\theta,x_r(t))\equiv 0\,.
\]

\vspace{-1.8em}
\noindent
Furthermore, by recalling (\ref{bound-W}), the equation of $\dot V_{est}$ can be elaborated by
\[\ba{rcl}
\dot V_{est} &\leq& \tilde\theta^{\top}\mbox{exp}\left(-\dst\int_{t}^{t+\delta}  \mathcal{H}(M_1(x_r(\tau))) d\tau\right)\tilde\theta -|\tilde\theta|^2 \,\\
&\leq& -(1-e^{-2\mu })|\tilde\theta(t)|^2\,\\
\ea\]
which proves (\ref{Vest-b})  with $a_3=(1-e^{-2\mu })$.

\vspace{-1.0em}
To show (\ref{Vest-c}), it is observed from (\ref{def-H}) that
\[
\frac{\partial H(\theta,\tilde\theta,x_r)}{\partial \tilde\theta} =-\varsigma(x_r)\frac{\partial\tilde\phi_s}{\partial\tilde\theta}(\theta,\tilde\theta,x_r) - k_{dz}\frac{\partial\mbox{dzv}}{\partial\tilde\theta}(\tilde\theta+\theta)\,.
\]

\vspace{-1.5em}
\noindent
Recalling (\ref{def-dz}), we can always find a constant $h>0$ such that
$
\left|{\partial H(\theta,\tilde\theta,x_r)}/{\partial \tilde\theta}\right|\leq h\,.
$
It is noted that
\beeq{\label{equ-sen}
\frac{\partial^2 \Phi(\tau,t,\tilde\theta)}{\partial \tilde\theta\partial\tau} = \frac{\partial H(\theta,\Phi(\tau,t,\tilde\theta),x_r(\tau))}{\partial \tilde\theta}\,
}

\vspace{-1.5em}
\noindent
which implies
\[
\left|{\partial \Phi(\tau,t,\tilde\theta)}/{\partial \tilde\theta}\right| \leq \left|{\partial \Phi(t,t,\tilde\theta)}/{\partial \tilde\theta}\right|e^{h(\tau-t)} = e^{h(\tau-t)}\,.
\]

\vspace{-1.5em}
\noindent
We also observe that
\[\ba{rl}
|\Phi(\tau,t,\tilde\theta)|^2 &\leq {V_{est}(\tau)}/{a_1}
\leq e^{-\frac{a_3}{a_2}(\tau-t)} {V_{est}(t)}/{a_1}\,\\
&\leq ({a_2}/{a_1})e^{-\frac{a_3}{a_2}(\tau-t)}|\tilde\theta|^2\,.
\ea\]

Therefore,
\[\ba{rcl}
\left|{\partial V_{est}}/{\partial \tilde\theta}\right| &=& \left|\dst\int_{t}^{t+\delta} 2\Phi^{\top}(\tau,t,\tilde\theta){\partial \Phi(\tau,t,\tilde\theta)}/{\partial \tilde\theta}\right|\,\\
&\leq& \dst\int_{t}^{t+\delta} 2|\Phi(\tau,t,\tilde\theta)|\cdot\left|{\partial \Phi(\tau,t,\tilde\theta)}/{\partial \tilde\theta}\right|d\tau\,\\
&\leq& \dst\int_{t}^{t+\delta} 2 \sqrt{{a_2}/{a_1}} e^{(h-\frac{a_3}{2a_2})(\tau-t)} d\tau\, |\tilde\theta|\,\\
&=& 2 \dst\sqrt{\frac{a_2}{a_1}} \frac{2a_2}{2ha_2-a_3}\left(e^{\frac{2ha_2-a_3}{2a_2}\delta}-1\right) |\tilde\theta| \,\\
&:=& a_4|\tilde\theta|\,.
\ea\]
The proof is thus completed. $\blacksquare$

\section{Proof of Theorem \ref{theo-general}}
\label{app-sec-proof-Th2}

\vspace{-1.0em}
Fix $\tau_{err}>1$ and let $\tau_{err}'$ and $\tau_{est}$  satisfy
\[
\tau_{err} > \tau_{err}' > \tau_{est} >1\,.
\]
Define $\gamma_{\tilde\theta,e}(s):=\sigma_{\tilde\theta,e}\circ\alpha_1^{-1}(s)$ for all $s\geq0$. Then, using (\ref{ineq-V-a}), we can rewrite (\ref{ineq-V-2-b})  as
\[
V_{est} \geq \gamma_{\tilde\theta,e}(V_{err}) \,\Longrightarrow\, \dot V_{est}\leq -\frac{a_3(\tau_{est}-1)}{\tau_{est}}|\tilde\theta|^2\,.
\]

\vspace{-1.8em}
\noindent
Since $\gamma_s(s)\leq s$, it can be seen that
\[
\tau_{err}'\gamma_s\left({s}/{\tau_{err}'}\right)\leq s\,,\quad\forall s\in\mathbb{R}_+\,.
\]

\vspace{-1.8em}
\noindent
Bearing in mind this fact and (\ref{alpha-3}), (\ref{V-err-1}) can be bounded from above as
\[\ba{l}
\dot V_{err} \leq -\tau_{err}\gamma_s\circ a^\ast l_{\gamma}\kappa_2\circ\alpha_1^{-1}\circ\alpha_2(|e|) \alpha_4(|e|)\kappa_1(|e|)\,\\
\qquad - \alpha_4(|e|)\kappa_1(|e|)\gamma_s(|\tilde\theta|)\,\\
\leq -(\tau_{err}-\tau_{err}')\gamma_s\circ a^\ast l_{\gamma}\kappa_2\circ\alpha_1^{-1}\circ\alpha_2(|e|) \alpha_4(|e|)\kappa_1(|e|)\,\\
\qquad -\tau_{err}'\gamma_s\circ a^\ast l_{\gamma}\kappa_2\circ\alpha_1^{-1}\circ\alpha_2(|e|) \alpha_4(|e|)\kappa_1(|e|)\,\\
\qquad + \tau_{err}'\gamma_s(\frac{|\tilde\theta|}{\tau_{err}'}) \alpha_4(|e|)\kappa_1(|e|)\,\\
\leq -\alpha_{err}(|e|)\,\\
-\tau_{err}'\left[\gamma_s\circ a^\ast l_{\gamma}\kappa_2\circ\alpha_1^{-1}\circ\alpha_2(|e|) - \gamma_s(\frac{|\tilde\theta|}{\tau_{err}'})\right]\alpha_4(|e|)\kappa_1(|e|)\,
\ea\]

\vspace{-1.8em}
\noindent
with $\alpha_{err}(s)=(\tau_{err}-\tau_{err}')\gamma_s\circ a^\ast l_{\gamma}\kappa_2\circ\alpha_1^{-1}\circ\alpha_2(s) \alpha_4(s)\kappa_1(s)$.
This, together with the right side of (\ref{ineq-V-a}), yields
\[\ba{l}
\tau_{err}'a^\ast l_{\gamma}\kappa_2\circ\alpha_1^{-1}(V_{err}) \geq |\tilde\theta| \,\Longrightarrow\,
\dot V_{err} \leq  -\alpha_{err}(|e|)\,.
\ea\]

\vspace{-1.8em}
\noindent
Recalling (\ref{Vest-a}), (\ref{sigma-21}), and (\ref{a-ast}), we further have
\[\ba{l}
V_{err} \geq \gamma_{e,\tilde\theta}(V_{est}) \,\Longrightarrow\,
\dot V_{err} \leq  -\alpha_{err}(|e|)
\ea\]
with $\gamma_{e,\tilde\theta}(s):=\alpha_1\circ \sigma_{\tilde\theta,e}^{\ominus}(\frac{\tau_{est}}{\tau_{err}'}s)$ for all $s\geq0$.

\vspace{-1.0em}
Note that $\gamma_{e,\tilde\theta}\in\mathcal{K}$ and $\gamma_{\tilde\theta,e}\in\mathcal{K}$, and
\[
\gamma_{e,\tilde\theta}\circ\gamma_{\tilde\theta,e} \leq ({\tau_{est}}/{\tau_{err}'}) \mbox{Id} < \mbox{Id}\,.
\]

\vspace{-1.5em}
\noindent
Therefore, using standard ISS small-gain theorem \cite{Jiang&Teel(1994)} or Theorem \ref{theo-csgt}, it is immediate to conclude Theorem \ref{theo-general}. $\blacksquare$

\section{Proof of Corollary \ref{Coro1}}
\label{app-sec-proof-Coro1}

Let $\tau_i>2$, $i=1,2$ be such that $\tau_{err}\geq\tau_{1}\tau_2$. Then with (\ref{alpha-3}), (\ref{V-err-1}) can be rewritten as
\[\ba{l}
\dot V_{err} \leq -\tau_{err}\gamma_s\circ a^\ast l_{\gamma}\kappa_2\circ\alpha_1^{-1}\circ\alpha_2(|e|) \alpha_4(|e|)\kappa_1(|e|)\,\\
\qquad - \alpha_4(|e|)\kappa_1(|e|)\gamma_s(|\tilde\theta|)\,\\
\leq -\tau_1\gamma_s\circ \tau_2 a^\ast l_{\gamma}\kappa_2\circ\alpha_1^{-1}\circ\alpha_2(|e|) \alpha_4(|e|)\kappa_1(|e|)\,\\
\qquad + \gamma_s(|\tilde\theta|) \alpha_4(|e|)\kappa_1(|e|)\,\\
\leq - (\rho_{err}(|e|)- \sigma_{err}(|\tilde\theta|))\zeta_{err}(|e|)
\ea\]
where for convenience we have defined
\[\ba{l}
\rho_{err}(s) := \tau_1\gamma_s\circ a^\ast \tau_2l_{\gamma}\kappa_2\circ\alpha_1^{-1}\circ\alpha_2(s)\,,\quad\\
\zeta_{err}(s) := \alpha_4(s)\cdot\kappa_1(s)\,,\quad \sigma_{err}(s) := \gamma_s(s)\,
\ea
\]
for all $s\geq0$, which yields
\[
\rho_{err}\circ\alpha_2^{-1}\circ\alpha_1(s) = \tau_1\gamma_s\circ a^\ast \tau_2l_{\gamma}\kappa_2(s)\,.
\]
On the other hand, let
\[\ba{c}
\underline{\alpha}_{est}(s) := a_1 s^2\,,\quad \bar{\alpha}_{est}(s) := a_2 s^2,\\
\zeta_{est}(s) := s\,,\quad
\rho_{est}(s) := a_3s\,,\quad
\sigma_{est}(s) := a_4l_{\gamma} \kappa_2(s)\,
\ea\]
which implies
\[
\underline{\alpha}_{est}^{-1}\circ\bar{\alpha}_{est}\circ\rho_{est}^{-1}\circ\tau_2\sigma_{est}(s) = a^\ast \tau_2l_{\gamma}\kappa_2(s)\,.
\]
It then can be easily deduced from the proof of Lemma \ref{lemma-non-ex} that the ISS Lyapunov function $V_{est}(t,\tilde\theta)$  for (\ref{tilde-theta-2}) fulfills
\beeq{\label{B-V_id}
\dot {V}_{est} \leq - (\rho_{est}(|\tilde\theta|)- \sigma_{est}(|e|))\zeta_{est}(|\tilde\theta|)\,.
}
Thus we choose the sum-type Lyapunov function $V_{cl}(t,e,\tilde\theta)$ as in (\ref{V-cl-1}) with
\[\ba{l}
\lambda_{err}(s):=\sigma_{est}\circ\alpha_1^{-1}(s)\cdot \zeta_{est}\circ\rho_{est}^{-1}\circ\tau_2\sigma_{est}\circ\alpha_1^{-1}(s)\,\\
\lambda_{est}(s):=\sigma_{err}\circ\underline{\alpha}_{est}^{-1}(s)\cdot \zeta_{err}\circ\rho_{err}^{\ominus}\circ\tau_1\sigma_{err}\circ\underline{\alpha}_{est}^{-1}(s)\,.
\ea\]
Computing the derivative of the Lyapunov function $V_{cl}$ yields
\[\ba{l}
\dot V_{cl} \leq -\dfrac{\tau_1-1}{\tau_1}\lambda_{err}(V_{err})\zeta_{err}(|e|)\rho_{err}(|e|) \,\\ -\dfrac{\tau_2-1}{\tau_2}\lambda_{est}(V_{est})\zeta_{est}(|\tilde\theta|)\rho_{est}(|\tilde\theta|)\,\\
- \left(\dfrac{1}{\tau_1}\rho_{err}(|e|)- \sigma_{err}(|\tilde\theta|)\right)\zeta_{err}(|e|)\lambda_{err}(V_{err})\,\\
- \left(\dfrac{1}{\tau_2}\rho_{est}(|\tilde\theta|)- \sigma_{est}(|e|)\right)\zeta_{est}(|\tilde\theta|)\lambda_{est}(V_{est})\,.
\ea\]
Observe that using the nonlinear scaling technique \cite{Sontag&Teei(1995)} and combining the two cases $\rho_{err}(|e|)\geq \tau_1 \sigma_{err}(|\tilde\theta|)$ and $\rho_{err}(|e|)< \tau_1 \sigma_{err}(|\tilde\theta|)$, yields
\[\ba{l}
- \left(\dfrac{1}{\tau_1}\rho_{err}(|e|)- \sigma_{err}(|\tilde\theta|)\right)\zeta_{err}(|e|)\lambda_{err}(V_{err})\,\\
\leq \sigma_{err}(|\tilde\theta|)\cdot \zeta_{err}\circ\rho_{err}^{\ominus}\circ\tau_1\sigma_{err}(|\tilde\theta|)\cdot\,\\ \qquad \cdot\lambda_{err}\circ\alpha_2\circ\rho_{err}^{\ominus}\circ\tau_1\sigma_{err}(|\tilde\theta|)\,\\
\leq \dfrac{1}{\tau_2}\lambda_{est}(V_{est})\zeta_{est}(|\tilde\theta|)\rho_{est}(|\tilde\theta|)
\ea\]
where the last inequality is obtained by using the inequalities
\[\ba{l}
\sigma_{err}(|\tilde\theta|)\cdot \zeta_{err}\circ\rho_{err}^{\ominus}\circ\tau_1\sigma_{err}(|\tilde\theta|) \leq \lambda_{est}(V_{est})\,\\
\lambda_{err}\circ\alpha_2\circ\rho_{err}^{\ominus}\circ\tau_1\sigma_{err}(|\tilde\theta|) \leq \dfrac{1}{\tau_2}\zeta_{est}(|\tilde\theta|)\rho_{est}(|\tilde\theta|)\,.
\ea\]
Mimicking the above analysis, we can obtain
\[\ba{l}
- \left(\dfrac{1}{\tau_2}\rho_{est}(|\tilde\theta|)- \sigma_{est}(|e|)\right)\zeta_{est}(|\tilde\theta|)\lambda_{est}(V_{est})\,\\
\leq \dfrac{1}{\tau_1}\lambda_{err}(V_{err})\zeta_{err}(|e|)\rho_{err}(|e|)\,.
\ea\]
Therefore, the derivative of $V_{cl}$ can be further elaborated as
\[\ba{l}
\dot V_{cl} \leq -\dfrac{\tau_1-2}{\tau_1}\lambda_{err}(V_{err})\zeta_{err}(|e|)\rho_{err}(|e|) \,\\ -\dfrac{\tau_2-2}{\tau_2}\lambda_{est}(V_{est})\zeta_{est}(|\tilde\theta|)\rho_{est}(|\tilde\theta|)\,\\
\ea\]
which completes the proof. $\blacksquare$

\section{Proof of Theorem \ref{theo-rgeneral}}
\label{app-sec-proof-Theo-rg}

Let $\tau_{est}$ and $\tau_{err}'$ be such that
\vspace{-1.0em}
\[
\tau_{err} > \tau_{err}' > \tau_{est} >1\,.
\vspace{-1.0em}
\]
Along the $\tilde\theta$-subsystem in (\ref{rclsys}) and using (\ref{kappa-34}), we have
\vspace{-1.0em}
\[\ba{rl}
\dst\dot {V}_{est} &\leq -a_3|\tilde\theta|^2 + a_4l_{\gamma} |\tilde\theta|\kappa_2(|e|) \,\\ & \qquad + a_4 |\tilde\theta||\varsigma(e+x_r)|\cdot|g(e+x_r)|\cdot|d(t)|\,\\
&\leq -a_3|\tilde\theta|^2 + a_4l_{\gamma} |\tilde\theta|\kappa_2(|e|) + a_4 |\tilde\theta|\sqrt{\nu(|e|)}|d(t)|\,\\ &\qquad +
a_4 (\kappa_3^\ast+\kappa_4^\ast) |\tilde\theta| |d(t)|\,\\
&\leq -a_3|\tilde\theta|^2 + a_4|\tilde\theta|\bar\kappa_2(|e|)  + a_4 |\tilde\theta|(|d(t)|^2 + \iota|d(t)|)
\vspace{-1.0em}
\ea\]
with $\iota=\kappa_3^\ast+\kappa_4^\ast$.
Similar to the proof of Theorem \ref{theo-general}, it can be seen that the $\tilde\theta$-subsystem is uniformly ISS with respect to state $\tilde\theta$ and inputs $e,d(t)$, and fulfills
\vspace{-1.0em}
\[\ba{l}
V_{est} \geq \max\{\bar\gamma_{\tilde\theta,e}(V_{est}),\sigma_{\tilde\theta,d}(|d(t)|)\} \,\Longrightarrow\, \\
\dot V_{est}\leq -\frac{a_3(\tau_{est}-1)}{2\tau_{est}}|\tilde\theta|^2\,
\vspace{-1.0em}
\ea\]
with
\[\ba{l}
\bar\gamma_{\tilde\theta,e}(s)=a_1(\tau_{est} a^\ast)^2(\bar\kappa_2\circ\alpha_1^{-1}(s))^2\,\\
\sigma_{\tilde\theta,d}(s) = 4a_1 \dfrac{(a^\ast\tau_{est})^2}{(\tau_{est}-1)^2}(s^2 + \iota s)^2\,.
\ea\]

Along the $e$-subsystem in (\ref{rclsys}) and recalling Assumption \ref{ass-ideal-sys}, we have
\vspace{-1.0em}
\[\ba{l}
\dot {V}_{err} \leq -\alpha_3(|e|) -k_d\left|\dst\frac{\partial V_{err}}{\partial e}g(e+x_r)\right|^2 \, \\ \qquad + \left|\dst\frac{\partial V_{err}}{\partial e}\right| \,|\tilde\phi_s(\theta,\tilde\theta,e+x_r)| + \left|\dst\frac{\partial V_{err}}{\partial e}g(e+x_r)\right|\,|d(t)|\,\\
 \leq -\alpha_3(|e|)  + \alpha_4(|e|)\kappa_1(|e|)\gamma_s(|\tilde\theta|) + \dfrac{1}{4k_d} |d(t)|^2\,
 \vspace{-1.0em}
\ea
\]
where  (\ref{r-eta}) and Young's inequality are used to obtain the second inequality.

Again, similar to the proof of Theorem \ref{theo-general}, it can be seen that the $e$-subsystem is uniformly ISS with respect to state $e$ and inputs $\tilde\theta,d(t)$, and fulfills
\vspace{-1.0em}
\[\ba{l}
V_{err} \geq \max\{\bar\gamma_{e,\tilde\theta}(V_{est}),\sigma_{e,d}(|d(t)|\} \,\Longrightarrow\,\\
\dot V_{err} \leq  -\bar\alpha_{err}(|e|)
\vspace{-1.0em}
\ea\]
with
\vspace{-1.0em}
\[\ba{l}
\bar\alpha_{err}(s) := \frac{\tau_{err}-\tau_{err}'}{2}\gamma_s\circ a^\ast \bar\kappa_2\circ\alpha_1^{-1}\circ\alpha_2(s) \alpha_4(s)\kappa_1(s)\,\\
\bar\gamma_{e,\tilde\theta}(s):=\bar\gamma_{\tilde\theta,e}^{\ominus}(\frac{\tau_{est}}{\tau_{err}'}s)\,\\
\sigma_{e,d}(s) := \alpha_2\circ\bar\alpha_{err}^{-1}(\frac{1}{4k_d} s^2)\,.
\vspace{-1.0em}
\ea\]

\vspace{-1.6em}
Therefore, according to Theorem \ref{theo-csgt},  system (\ref{rclsys}) is uniformly ISS with respect to the state $(e,\tilde\theta)$ and input $d(t)$. $\blacksquare$

\section{Proof of Theorem \ref{theo-f}}
\label{app-proof-theo-f}

With (\ref{alpha-3-f}), mimicking the proof of Theorem \ref{theo-general}, we can conclude that both the $\tilde\theta$ and $e$ subsystems are uniformly ISS. More explicitly, along (\ref{clsys-f}), by letting $V_{\epsilon}=|\epsilon_e|^2$, we have
\[\ba{l}
V_{err} \geq \check\gamma_{e,\tilde\theta}(V_{est}) \,\Longrightarrow\,
\dot V_{err} \leq  -\check\alpha_{err}(|e|)\,\\
V_{est} \geq \max\{\check\gamma_{\tilde\theta,e}(V_{err}),\check\gamma_{\tilde\theta,\epsilon}(V_{\epsilon})\} \,\Longrightarrow\,
\dot V_{err} \leq  -\check a_{est}|\tilde\theta|^2
\ea\]
for some appropriately defined function $\check\alpha_{err}\in\mathcal{K}_{\infty}$ and positive constant $\check a_{est}>0$.

%
%

Then we compute the derivative of $V_{\epsilon}$ along (\ref{clsys-f}) as
\[\ba{l}
\dot V_{\epsilon} \leq -2\epsilon_e^\top K(\epsilon_e) + 2 |\epsilon_e|(\varrho_1(|e|) + \varrho_1^\ast)\gamma_s(|\tilde\theta|)\,\\
\leq -2\epsilon_e^\top K(\epsilon_e)+ |\epsilon_e|\varrho_1(|e|)^2 + |\epsilon_e|(\varrho_1^\ast\gamma_s(|\tilde\theta|) +\frac{1}{4}\gamma_s(|\tilde\theta|)^2) \,\\
\leq -2\epsilon_e^\top K(\epsilon_e) + |\epsilon_e|\pi_{\epsilon,e}(V_{err}) + |\epsilon_e|\pi_{\epsilon,\tilde\theta}(V_{est}) \,\\
\leq -2k_{\epsilon}|\epsilon_e|^2 -  2|\epsilon_e|\left[\pi_{\epsilon,e}\circ\check\gamma_{e,\tilde\theta}\circ\check\gamma_{\tilde\theta,\epsilon}(V_\epsilon) - \frac{1}{2}\pi_{\epsilon,e}(V_{err}) \right]\\
\qquad -  2|\epsilon_e|\left[\pi_{\epsilon,\tilde\theta}\circ\check\gamma_{\tilde\theta,\epsilon}(V_\epsilon) - \frac{1}{2}\pi_{\epsilon,\tilde\theta}(V_{est}) \right]\\
\leq -2k_{\epsilon}|\epsilon_e|^2 -  2|\epsilon_e|\left[\pi_{\epsilon,e}\circ\check\gamma_{e,\tilde\theta}\circ\check\gamma_{\tilde\theta,\epsilon}(V_\epsilon) - \frac{1}{2}\pi_{\epsilon,e}(V_{err}) \right]\\
\qquad -  2|\epsilon_e|\left[\pi_{\epsilon,\tilde\theta}\circ\check\gamma_{\tilde\theta,\epsilon}(V_\epsilon) - \frac{1}{2}\pi_{\epsilon,\tilde\theta}(V_{est}) \right]\,.
\ea\]
Recalling the definition of $\pi_{\epsilon,\tilde\theta}$ in (\ref{pi-theta}), it can be easily deduced that $\frac{1}{2}\pi_{\epsilon,\tilde\theta}(s) \leq \pi_{\epsilon,\tilde\theta}(\mu' s)$
holds for some $0<\mu'<1$.
With this being the case, we have
\[\ba{l}
V_{\epsilon} \geq \max\{\check\gamma_{\epsilon,e}(V_{err}),\check\gamma_{\epsilon,\tilde\theta}(V_{est})\} \,\Longrightarrow\,
\dot V_{err} \leq  -2k_{\epsilon}|\epsilon_e|^2
\ea\]
with
\[\ba{l}
\check\gamma_{\epsilon,e}(s) := \check\gamma_{\tilde\theta,\epsilon}^{-1}\circ\check\gamma_{e,\tilde\theta}^{-1}\circ\pi_{\epsilon,e}^{-1}\circ\frac{1}{2}\pi_{\epsilon,e}(s)\,\\
\check\gamma_{\epsilon,\tilde\theta}(s) := \check\gamma_{\tilde\theta,\epsilon}^{-1}(\mu's)\,.
\ea\]
Observe that the extended system  (\ref{clsys-f}) is comprised of three simple cycles as
\[\ba{l}
e \rightarrow \tilde\theta \rightarrow e\,\\
\tilde\theta \rightarrow \epsilon_e \rightarrow \tilde\theta\,\\
e \rightarrow \tilde\theta \rightarrow \epsilon_e \rightarrow e\,.
\ea
\]
According to Theorem \ref{theo-csgt}, we verify the small-gain condition (\ref{csmc}) for all these cycles as
\[\ba{l}
\check\gamma_{e,\tilde\theta}\circ\check\gamma_{\tilde\theta,e} =\dst\frac{\tau_{est}}{\tau_{err}'} \mbox{Id} < \mbox{Id}\,\\
\check\gamma_{\tilde\theta,\epsilon}\circ\check\gamma_{\epsilon,\tilde\theta} = \mu' \mbox{Id} < \mbox{Id}\,\\
\check\gamma_{e,\tilde\theta}\circ\check\gamma_{\tilde\theta,\epsilon}\circ\check\gamma_{\epsilon,e} = \pi_{\epsilon,e}^{-1}\circ\frac{1}{2}\pi_{\epsilon,e}< \mbox{Id}\,,
\ea\]
which in turn shows the theorem. $\blacksquare$

\section{Proof of Lemma \ref{lemma-SEA}}
\label{app-sec-proof-Pro1}

According to Lemmas \ref{lemma-1} and \ref{lemma-non-ex}, it is clear that there are two issues to be verified in order to finish the proof: (i) Assumption \ref{ass-PE-2} is satisfied and (ii) the function $\kappa_2(\cdot)$ of Lemma \ref{lemma-1} in the current setting satisfies $\kappa_{2}(s)\in\mathcal{O}(s)$ as $s\rightarrow 0^+$. The first issue guarantees that there exists $k_{dz}$ such that the $\tilde\theta$-subsystem with state $\tilde\theta$ and input $e_1$ is uniformly ISS with some gain function $\gamma_{\tilde\theta,1}$, while the latter guarantees that the resulting gain function $\gamma_{\tilde\theta,1}$ satisfies  $\gamma_{\tilde\theta,1}(s)\in\mathcal{O}(s)$ as $s\rightarrow 0^+$.

In light of these observations, we address the first issue and observe that
\vspace{-1.0em}
\[
\varsigma(d_r)\frac{\partial \phi(\theta,d_r)}{\partial\theta} = b_1  e^{b_2\theta_1} |d_r|^{\theta_2+p^\ast+2} \begin{pmatrix} b_2^2 & b_2 \log|d_r| \cr b_2 \log|d_r| & (\log|d_r|)^2\end{pmatrix}\,.
\vspace{-1.0em}
\]
This in turn indicates that (\ref{ineq-ND}) in Assumption \ref{ass-PE-2} is satisfied with the above choice of $\varsigma$ and matrix $M_1$ of the form
\vspace{-1.0em}
\[
M_1(d_r) = M_0\begin{pmatrix} b_2^2 & b_2 \log|d_r| \cr b_2 \log|d_r| & (\log|d_r|)^2\end{pmatrix}
\vspace{-1.0em}
\]
with an appropriately defined $M_0>0$.
Moreover, with $d_r(t)=\mbox{exp}(\sin t)$, the expression of $M_1(d_r(t))$ can be further elaborated as
\vspace{-1.0em}
\[
M_1(d_r(t)) = M_0\begin{pmatrix} b_2^2 & b_2 \sin t \cr b_2 \sin t & \sin^2 t\end{pmatrix}\,.
\]

\vspace{-2.1em}
Simple calculations then show that the PE condition (\ref{ineq-PE}) is also satisfied, which in turn indicates that Assumption \ref{ass-PE-2} is verified. In this way, according to Lemma \ref{lemma-non-ex}, we can find a $k_{zd}^\ast>0$ such that for all $k_{zd}>k_{zd}^\ast$, the above $\tilde\theta$ dynamics permits a uniform ISS Lyapunov function $V_{\tilde\theta}$ as in (\ref{V-1}), fulfilling (\ref{SEA-V-theta}) with the ISS gain function $\gamma_{\tilde\theta,1}$ of the form
\vspace{-1.0em}
\[
\gamma_{\tilde\theta,1} = a_1(\tau_{est} a^\ast l_{\gamma})^2 \kappa_2(\sqrt{s})^2\,.
\]

\vspace{-2.0em}
We now proceed to address the second issue, i.e., to verify the property of function $\kappa_2$, which is defined in Lemma \ref{lemma-1}. Note that the function $\varphi$ in Lemma \ref{lemma-1} is smooth in both $\theta$ and $e$ in the current setting. Hence, the resulting $\kappa_2$ takes the form  $\kappa_2(s) = s \kappa'_2(s)$
for some $\kappa_2'\in\mathcal{SN}$, which indicates $\gamma_{\tilde\theta,1}(s)\in\mathcal{O}(s)$ as $s\rightarrow 0^+$, i.e., there exists a function $\bar\gamma_{\tilde\theta,1}\in\mathcal{SN}$ such that $\gamma_{\tilde\theta,1}(s)\leq \bar \gamma_{\tilde\theta,1}(s) \cdot s$ for all $s\in\mathbb{R}_+$. The proof is thus completed. $\blacksquare$


\begin{thebibliography}{99}
\bibitem{Sastry&Bodson(1989)}
S. Sastry and M. Bodson, \emph{Adaptive Control: Stability, Convergence and Robustness.} London: Prentice-Hall, 1989.

\bibitem{Krstic&PK(1995book)}
M. Krstic, I. Kanellakopoulos, and P. K. Kokotovic, \emph{Nonlinear and Adaptive Control Design.} New York: Wiley, 1995.

\bibitem{Mario&Tomei(1995book)}
R. Marino and P. Tomei, \emph{Nonlinear Control Design. Geometric, Adaptive and Robust.} Upper Saddle River, NJ: Prentice-Hall, 1995.

\bibitem{Ioannou&Sun(2012)}
P.A. Ioannou and J. Sun, \emph{Robust Adaptive Control}. Courier Corporation, 2012.

\bibitem{Astolfi&Karagiannis(book2008)}
A. Astolfi, D. Karagiannis, and R. Ortega, \emph{Nonlinear and Adaptive Control with Applications}. Communications and Control Engineering, Springer, 2008.

\bibitem{Zhao&Xian(TIE2014)}
B. Zhao, B. Xian, Y. Zhang, and X. Zhang, ``Nonlinear robust adaptive tracking control of a quadrotor UAV via immersion and invariance methodology," \emph{IEEE Transactions on Industrial Electronics}, vol.62, no.5, pp.2891-2902, 2014.

\bibitem{Chen&Astolfi(cdc18)}
K. Chen and A. Astolfi, ``I\&I adaptive control for systems with varying parameters," in \emph{Proceedings of the 57th IEEE Conference on Decision and Control}, pp. 2005-2010, 2018.

\bibitem{Mazenc&Queiroz(TAC2013)}
F. Mazenc, M. Queiroz, and M. Malisoff, ``Uniform global asymptotic stability of a class of adaptively controlled nonlinear systems," \emph{IEEE Trans. Autom. Contr.}, vol.54, no.5, pp.1152-1158, 2013.

\bibitem{Wang&Chen&Xu}
X. Wang, Z. Chen, and D. Xu, ``A framework for global robust output regulation of nonlinear lower triangular systems with uncertain exosystems," \emph{IEEE Trans.
Autom. Contr.,} vol.63, no.3, pp.894-901, 2017.


\bibitem{Ito&Kellett(2018)}
H. Ito and C. M. Kellett, ``A small-gain theorem in the absence of strong iISS," \emph{IEEE Trans.
Autom. Contr.,} vol.64, no.9, pp.3897-3904, 2018.

\bibitem{Wang&Chris(cdc19)}
L. Wang and C. Kellett, ``Adaptive tracking control via immersion and invariance : An (i)ISS perspective," in \emph{Proceedings of the 58th IEEE Conference on Decision and Control}, 2019.

\bibitem{Astolfi&Ortega(2003)}
A. Astolfi and R. Ortega, ``Immersion and invariance: A new tool for stabilization and adaptive control of nonlinear systems," \emph{IEEE Trans. Autom. Contr.}, vol. 48, no. 4, pp. 590-606, 2003.

\bibitem{Marino&Tomei(1993)}
R. Marino and P. Tomei, ``Global adaptive output-feedback control of nonlinear systems, part ii: Nonlinear parameterization," \emph{IEEE Trans.
Autom. Contr.,} vol. 38, no. 1, pp. 33-48, Jan. 1993.


\bibitem{Annaswamy&Skantze(1998)}
A. Annaswamy, F. P. Skantze, and A. P. Loh, ``Adaptive control of continuous-time systems with convex/concave parametrizations," \emph{Automatica},
vol. 34, pp. 33-49, 1998.

\bibitem{Ge1999}
S. S. Ge, C. C. Hang, and T. Zhang, ``A direct adaptive controller for dynamic systems with a class of nonlinear parameterizations," \emph{Automatica}, vol. 35, pp. 741-747, 1999.

\bibitem{Liu&Ortega(2010TAC)}
X. Liu, R. Ortega, H. Su, and J. Chu, ``Immersion and invariance adaptive control of nonlinearly parameterized nonlinear systems," \emph{IEEE Trans.
Autom. Contr.}, vol.55, no.9, pp.2209-2214, 2010.

\bibitem{Wang&Ortega(2015TAC)}
L. Wang, R. Ortega, H. Su, and Z. Liu, ``Stabilization of nonlinear systems nonlinearly depending on fast time-varying parameters: An immersion and invariance approach," \emph{IEEE Trans.
Autom. Contr.}, vol.60, no.2, pp.559-564, 2015.

\bibitem{Lin(2002)}
W. Lin, ``Adaptive control of nonlinearly parameterized systems: The smooth feedback case," \emph{IEEE Trans.
Autom. Contr.,} vol. 47, no. 8, pp.
1249-1266, Aug. 2002.

\bibitem{Sontag-SCL1998-Integral-Variants}
E. D. Sontag, ``Comments on integral variants of ISS," \emph{Syst. and Control Lett.}, vol. 34, pp.93-100, 1998.

\bibitem{Kellett&Wirth(TAC2016)}
C. M. Kellett and F. Wirth, ``Nonlinear scaling of (i) ISS-Lyapunov functions," \emph{IEEE Trans.
Autom. Contr.,} vol.61, no.4, pp. 1087-1092, 2016.

\bibitem{Sontag(bookchapter2008)}
E. D. Sontag, ``Input-to-state stability: Basic concepts and results," in Nonlinear and Optimal Control Theory, A. Agrachev, A. Morse, E.
Sontag, H. Sussmann, and V. Utkin, Eds. Berlin, Germany: Springer-Verlag, vol. 1932, pp. 163-220, 2008.

\bibitem{Sontag(tac1989)}
E. D. Sontag, ``Smooth stabilization implies coprime factorization," \emph{IEEE Trans.
Autom. Contr.,} vol. 34, no.4, pp. 435-443, 1989.


\bibitem{Ito(CDC2002)}
H. Ito,  ``A constructive proof of ISS small-gain theorem using generalized scaling." In \emph{Proceedings of the 41st IEEE Conference on Decision and Control}, pp.2286-2291, 2002.

\bibitem{LiuHillJiang(Aut2011)}
T. Liu, D. J. Hill, and Z. P. Jiang, ``Lyapunov formulation of ISS cyclic-small-gain in continuous-time dynamical networks," \emph{Automatica}, vol.47, pp. 2088-2093, 2011.

\bibitem{Liu&Ortega(2011SCL)}
X. Liu, R. Ortega, H. Su, and J. Chu, ``On adaptive control of nonlinearly parameterized nonlinear systems: Towards a constructive procedure," \emph{Syst. and Control Lett.}, vol.60, pp.36-43, 2011.

\bibitem{SontagWang(1995SCL)}
E. D. Sontag and Y. Wang, ``On characterizations of the input-to-state stability property," \emph{Syst. and Control Lett.}, vol. 24, pp. 351-359, 1995.

\bibitem{Ito&Jiang(TAC2013)}
H. Ito, Z. Jiang, S. N. Dashkovskiy, and B. S. R\"uffer, ``Robust stability of networks of iISS systems: Construction of sum-type Lyapunov functions," \emph{IEEE Trans.
Autom. Contr.,} vol. 57, no. 5, pp.
1192-1207, 2013.


\bibitem{Dash&Ruffer(SIAM2010)}
S. Dashkovskiy, B. R\"uffer, and F. Wirth, ``Small gain theorems for large scale systems and construction of ISS Lyapunov functions," \emph{SIAM J. Control Optim.}, vol. 48, no. 6, pp.4089-4118, 2010.

\bibitem{Jiang&Teel(1994)}
Z. Jiang, A. Teel, and L. Praly, ``Small-gain theorem for ISS systems and applications," \emph{Math. Control. Sig. Syst.,}  no. 7, pp.
95-120, 1994.

\bibitem{Karagiannis&Sassano(AUT2009)}
D. Karagiannis, M. Sassano, and A. Astolfi, ``Dynamic scaling and observer design with application to adaptive control," \emph{Automatica}, vol.45, pp. 2883-2889, 2009.










%









\bibitem{Kellett&Dower(TAC2016)}
C. M. Kellett, and P. M. Dower, ``Input-to-state stability, integral input-to-state Sstability, and $\mathcal{L}_2$-gain properties: qualitative equivalences and interconnected systems," \emph{IEEE Trans.
Autom. Contr.,} vol.61, no.1, pp. 3-17, 2016.




\bibitem{Sontag&Teei(1995)}
E. D. Sontag, and A. Teel, ``Changing supply functions in input/state stable systems," \emph{IEEE Trans.
Autom. Contr.,} vol. 40, pp. 1476-1478, 1995.



\bibitem{Khalil(book2002)}
H. Khalil, \emph{Nonlinear Systems}, 3rd ed. Upper Saddle River, NJ: Prentice-Hall, 2002.


\bibitem{Stulov(2003)}
A. Stulov, ``Experimental and theoretical studies of piano hammer," in \emph{Proceedings of the Stockholm Musical Acoustics Conference}, vol. I, pp. 175-178, 2003.


\bibitem{Boisseau(2012)}
 S. Boisseau, G. Despesse, and B. A, Seddik, ``Adjustable nonlinear springs to improve efficiency of vibration energy harvesters," \emph{arXiv:1207.4559}, 2012.



\bibitem{WangKellett2019}
L. Wang, and C. M. Kellett, ``Adaptive semiglobal nonlinear output regulation: An extended-state observer approach," \emph{IEEE Trans.
Autom. Contr.,} vol. 65, no. 6, pp. 2670-2677, 2020.

\bibitem{EfimovBarabanov&Ortega(cdc18)}
D. Efimov, N. Barabanov, and R. Ortega, ``Robust stability under relaxed persistent exicitation conditions," in \emph{Proceedings of the 57th IEEE Conference on Decision and Control},  2018.




\end{thebibliography}
\end{document}